\documentclass[11pt,a4paper,reqno]{article}
\usepackage{amsmath}
\usepackage{amsthm}
\usepackage{amssymb}
\usepackage{mathrsfs}

\usepackage{verbatim}
\usepackage{url}
\usepackage[latin1]{inputenc}

\usepackage{caption}
\usepackage{graphicx,color}
\input labelfig.tex

\usepackage{enumerate}
\usepackage[inline]{enumitem}
\def\N{{\mathbb N}}
\def\Z{{\mathbb Z}}

\def\R{{\mathbb R}}
\def\C{{\mathbb C}}
\def\Eb{{\mathbb E}}

\def\Ws{{\mathscr W}}

\def\Bin{\textup{Bin}}
\def\LR{\mathsf{LR}}
\def\TG{\mathbb{T}}
\def\llra{\longleftrightarrow}
\def\Quad{\mathcal Q}

\def \P {{\bf P}}
\def\md{\mid}
\def\Bb#1#2{{\def\md{\bigm| }#1\bigl[#2\bigr]}}
\def\BB#1#2{{\def\md{\Bigm| }#1\Bigl[#2\Bigr]}}
\def\Bs#1#2{{\def\md{\mid}#1[#2]}}
\def\Pb{\Bb\P}
\def\Eb{\Bb\E}

\def\Pso#1{\Bs{\P_{#1}}}
\def\Pbo#1{\Bb{\P_{#1}}}
\def\PBo#1{\BB{\P_{#1}}}

\def\Ebo#1{\Bb{\E_{#1}}}
\def\EBo#1{\BB{\E_{#1}}}

\def\be{\begin{equation}}
\def\ee{\end{equation}}
\def\bea{\begin{equation*}}
\def\eea{\end{equation*}}

\def\eps{\varepsilon}

\def\Pr{{\mathbb P}}
\DeclareMathOperator{\E}{{\mathbb E}}
\DeclareMathOperator{\Var}{Var}

\DeclareMathOperator{\Inf}{Inf}

\newtheorem{thm}{Theorem}%[section]
\newtheorem{lma}[thm]{Lemma}

\newtheorem{prop}[thm]{Proposition}

\newtheorem*{df}{Definition}

\newtheorem{conj}[thm]{Conjecture}

\theoremstyle{remark}
\newtheorem*{remark}{Remark}
\newtheorem{preex}[thm]{Example}

\newtheorem*{notation}{Notation}
\newtheorem*{keywords}{Keywords}

\theoremstyle{definition}

\begin{document}

\title{Scaling limits for the threshold window:\\ When 
does a monotone Boolean function flip its outcome?}
\date{With an appendix by G\'abor Pete}
\author{Daniel Ahlberg and Jeffrey E. Steif}

\maketitle

\begin{abstract}
Consider a monotone Boolean function $f:\{0,1\}^n\to\{0,1\}$ and the canonical monotone coupling 
$\{\eta_p:p\in[0,1]\}$ of an element in $\{0,1\}^n$ chosen according to product measure with intensity 
$p\in[0,1]$. The random point $p\in[0,1]$ where $f(\eta_p)$ flips from $0$ to $1$ is often concentrated 
near a particular point, thus exhibiting a threshold phenomenon. For a sequence of such Boolean functions, 
we peer closely into this threshold window and consider, for large $n$, the limiting distribution 
(properly normalized to be nondegenerate) of this random point where the Boolean function
switches from being 0 to 1. We determine this distribution for a number of the Boolean functions which are 
typically studied and pay particular attention to the functions corresponding to iterated majority and percolation crossings. It turns out that these limiting distributions have quite varying behavior.
In fact, we show that any nondegenerate probability measure on $\R$ arises in this way for some sequence of Boolean functions.

\begin{keywords}
Boolean functions; sharp thresholds; influences; iterated majority function; near-critical percolation
\end{keywords}
%\begin{msc}
%60-XX
%\end{msc}
\end{abstract}

\section{Introduction}

It has been known for quite some time that typical events involving many independent random variables exhibit ``thresholds" in the sense that the probability of the given event changes sharply as the parameter of the independent random variables varies. Observations of this kind were first made in the context of random graphs by Erd\H{o}s and R\'enyi~\cite{erdren60}. A more general understanding of the existence of threshold phenomena has since then been obtained through a series of papers. For instance, Russo~\cite{russo82} showed that a monotone event defined in terms of a family of independent Bernoulli variables exhibits a threshold if its dependence on each variable is small. Russo's result was later refined by Talagrand~\cite{talagrand94}. The first estimates on the ``sharpness" of the threshold were obtained by Friedgut and Kalai~\cite{frikal96}, critically building on work originating from Kahn, Kalai and Linial~\cite{kahkallin88}. Related results also appeared in~\cite{boltho87},~\cite{friedgut99} 
and elsewhere; see also~\cite{kalsaf06} for a more extensive overview of the field.

%The mentioned studies contribute to understanding the occurrence of thresholds.
Less is known when it comes to closer inspections of the ``threshold window". Although the windows corresponding to certain graph properties are well understood, there is to our knowledge no general study of this transition. We aim with the present paper to offer a unified perspective on threshold transitions, and show that these transitions present quite varying behavior.

Let $f:\{0,1\}^n\to\{0,1\}$ be a monotone (increasing) Boolean function and assign $[0,1]$-uniform random 
variables $\xi_1,\xi_2,\ldots,\xi_n$ to the elements of $[n]:=\{1,2,\ldots,n\}$. For $p\in[0,1]$, let 
$\eta_p=\{i\in[n]:\xi_i\le p\}$, and note that
\begin{enumerate*}[label=\itshape\alph*)]
\item each $i\in[n]$ is present in $\eta_p$ with probability $p$ independently for different $i$; and
\item $i\in \eta_p$ implies $i\in \eta_{p'}$ for $p'>p$.
\end{enumerate*}
Thus, $\eta_p$ corresponds to an element $\omega\in\{0,1\}^n$ chosen according to 
product measure with intensity $p$, in the sequel denoted by $\Pr_p$, and $(\eta_p)_{p\in[0,1]}$ constitutes the 
standard monotone coupling of elements in $\{0,1\}^n$ chosen according to $\Pr_p$, as $p$ varies between 0 and 
1. We study the random point $p$ at which $f(\eta_p)$ changes from 0 to 1. 
For a given sequence of monotone Boolean functions $(f_n)_{n\ge 1}$,
our goal will be to find the (nondegenerate) limiting distribution of this random point after 
proper normalization, should it exist. In the most interesting examples, one has a threshold 
phenomenon where, for large $n$, $\Pr_p(f_n)$ goes from 0 to 1 within a very small interval, which 
results in this transition point having a degenerate limit; one then needs to renormalize
in order to obtain a nondegenerate limiting distribution.
(If there is no threshold phenomenon, then this point should already have a nondegenerate limit and no
further analysis is made.) One of
our goals is to describe the distribution of this random point for some commonly studied Boolean functions.

To be more precise, given a monotone Boolean function $f:\{0,1\}^n\to\{0,1\}$, we define the random variable
$$
T(f):=\min\{p\in[0,1]:f(\eta_p)=1\};
$$
this is the point where $f$ switches from 0 to 1 in the canonical coupling. Given a sequence $(f_n)_{n\ge1}$ 
of monotone Boolean functions $f_n:\{0,1\}^n\to\{0,1\}$, we want to find,
if possible, normalizing constants $(a_n)_{n\ge1}$ and $(b_n)_{n\ge1}$ with the $a_n$'s nonnegative
such that $a_n(T(f_n)-b_n)$ converges, as $n\to\infty$, to a nondegenerate limiting distribution, 
and to determine what that limit may be. Observe that for $x\in\R$
\be\label{eq:identity}
\Pr\big(a_n(T(f_n)-b_n)\le x\big)\,=\,\Pr\big(T(f_n)\le b_n+x/a_n\big)\,=\,\Pr_{p_n}\big(f_n(\omega)=1\big),
\ee
where $p_n=b_n+x/a_n$. Recall the theorem of types (see \cite[Theorem~3.7.5]{durrett10}) 
which tells us that there is {\em essentially} only one
way to normalize a sequence of random variables and that there is {\em essentially} at most one possible
nondegenerate limiting distribution. (The word ``essentially'' in the latter part of the statement
means ``up to a change of variables of the form $x\mapsto ax+b$'').

\begin{notation}
When understood from the context, we will write $T_n$ for $T(f_n)$.
\end{notation}

The dictatorship function, which for every $n$ outputs the value of the first coordinate of the input, clearly
has that $T_n$ is uniformly distributed on the interval $[0,1]$ for each $n$ and so no scaling is needed. 
A simple example where scaling is needed is the OR function which is 1 if and only if at least one bit is 1.
In this case, it is immediate to check that 
$nT_n$ converges in distribution to a unit exponential random variable.
The cases when nontrivial normalization is needed are exactly those covered in the next definition.

\begin{df}\label{d.threshold}
We say that $(f_n)_{n\ge1}$ has a {\bf threshold} if there exists a sequence $(p_n)_{n\ge1}$ such that for all $\eps>0$,
$$
\lim_{n\to\infty} \Pr_{p_n+\eps}(f_n=1)=1\quad \text{and}\quad
\lim_{n\to\infty} \Pr_{p_n-\eps}(f_n=1)=0.
$$
This is equivalent to $T_n-p_n$ approaching $0$ in distribution. (Often $p_n$ will not depend on $n$.)
\end{df}

We give an alternative description of $T(f)$ which will be useful to have in mind, in particular when
we study the Boolean function known as ``tribes''. Recall that a \emph{$1$-witness} for $f$
is a minimal set $W\subseteq[n]$ such that $\{\omega_i=1\text{ for all }i\in W\}$ implies $f(\omega)=1$. 
Similarly, a \emph{$0$-witness} for $f$ is a minimal set $W\subseteq[n]$ 
such that $\{\omega_i=0\text{ for all }i\in W\}$ implies $f(\omega)=0$. Witnesses may be used to characterize $T(f)$. 
Writing $\Ws^1$ for the set of 1-witnesses and $\Ws^0$ for the set of 0-witnesses for $f$, it is immediate to check that for any monotone Boolean function one has
\be\label{eq:WitnessCharacterization}
T(f)\,=\,\min_{W\in\Ws^1}\,\max_{i\in W}\,\xi_i\,=\,\max_{W\in\Ws^0}\,\min_{i\in W}\,\xi_i.
\ee

We next briefly discuss what one should {\em expect} the normalizing constants
$(a_n)_{n\ge1}$ and $(b_n)_{n\ge1}$ to be in typical situations. Certainly it is reasonable that
$b_n$ should be close to $\E[T_n]$. In cases where we have a threshold, heuristically, the 
size of the ``threshold interval" around $p_n$ where $\Pr_{p}(f_n=1)$ moves from being near 0 to 
being near 1 \emph{should} be governed by $\frac{d}{dp}\Pr_{p}(f_n=1)$ evaluated at $p=p_n$.
The Margulis-Russo formula (see e.g.~\cite[Theorem~III.1]{garste12})
tells us that this is equal to the total influence at $p_n$ (defined below). Therefore the total influence dictates what the scaling factor $a_n$ \emph{should} be. We mention that while this heuristic
for the scaling works in most natural examples, it is certainly not true in general. For example,
if $f_n$ is the Boolean function which is the AND of majority (to be defined later) on $n$ bits and
dictator, then the total influence will be of order $\sqrt{n}$ but no scaling ($a_n\equiv1$) is needed to obtain
a nondegenerate limit for $T_n$.

\begin{df}\label{d.influence}
Given a Boolean function  $f$ of $n$ variables and a variable $i\in [n]$, we say that  
$i$ is {\bf pivotal} for $f$ for $\omega$ if $f(\omega)\neq f(\omega^i)$
where $\omega^i$ is $\omega$ but flipped in the $i$th coordinate. 
The {\bf influence} of the $i$th bit with respect to $p$, denoted by $\Inf^p_i(f)$, is defined by 
\[
\Inf^p_i(f):=\Pr_p(\mbox{ $i$ is pivotal for $f$ }) 
\]
and the {\bf total influence} with respect to $p$ is defined to be $\sum_{i\in[n]} \Inf^p_i(f)$.
\end{df}

\subsection{Limiting behavior for some specific Boolean functions}

We now summarize some of the results that we will obtain. The paper will begin by analyzing
the limiting distribution of $T_n$ for the majority function (which will be normal),
the tribes function (which will be a reverse Gumbel distribution) and certain properties
associated to graphs, such as connectivity and clique containment. A connection between the
tribes function and the coupon collector problem is discussed. These results, which are not difficult,
are presented in Section \ref{s.majority}.%, \ref{s.tribes} and \ref{s.graphs}. 

A class of functions that offers a quite interesting analysis is the so-called iterated majority functions.
%Next, for the class of so-called iterated majority functions,
For this class the analysis of the limiting distribution
of $T_n$ (both its existence and its properties) requires somewhat more work and involves
dynamical systems.
%This will be done in Section \ref{s.IM}. 
%The precise definition of iterated majority functions is as follows. 
Given an odd integer $m\ge3$, the iterated $m$-majority function is defined recursively on $m^n$ bits 
as follows. One  constructs an $m$-ary tree of height $n$ and places 0's and 1's at the leaves. 
One takes the majority of the bits in each {\em family} of $m$ leaves and thus obtains 0 and 1 values 
for the nodes at height $n-1$. One then continues iteratively until the root is assigned a value. 
This is defined to be the output of the function.
%which we denote by $f_n$ (where $m$ is implicit).

The iterated majority function has been studied in various papers and
is of interest in, among other areas, theoretical computer science.
Here we simply mention one paper, by Mossel and O'Donnell~\cite{mosodo03}, where these
functions are explicitly studied. These authors showed that this
family, as $m$ varies, provides examples of sequences of monotone Boolean functions
where the ``noise sensitivity exponent" (which we do not define here) is
arbitrarily close to $1/2$.

%We do not determine the limiting distribution for iterated majority explicitly, and possibly there is no explicit relation to known distributions.
In Section \ref{s.IM} we identify the precise rate of decay for the tails of the limiting distributions for this class of Boolean functions.
To state the result of the analysis for iterated majority, we let, for odd integers $m\ge3$,  
$$
\gamma(m):=\,m\,\binom{m-1}{\frac{m-1}{2}}\,2^{-(m-1)}\quad \text{and}\quad
\beta(m):=\,\frac{\log\frac{m+1}{2}}{\log\gamma(m)}.
$$

\begin{thm}\label{thm:m-majority} 
Consider, for each odd integer $m\ge3$, iterated $m$-ary majority on $m^n$ bits. 
\begin{enumerate}[label=\itshape\alph*),topsep=3pt,partopsep=1pt,itemsep=3pt,parsep=1pt]
\item
%For every odd integer $m\ge3$,
Then $\gamma(m)^n(T_n-\frac{1}{2})$ converges in distribution, as $n$ tends to infinity, to a random variable whose distribution $F_m$ is symmetric, absolutely continuous and fully supported on $\R$. Moreover, for $m\neq m'$, 
$F_m$ and $F_{m'}$ are not related by a linear change of variables.
\item
%For each odd integer $m\ge3$, 
There exist constants $c_1=c_1(m)$ and $c_2=c_2(m)$ in $(0,\infty)$ so that for all $x\ge 1$,
$$
\exp(-c_1x^{\beta(m)})\,\le\,\Pr(W_m\ge x)\,\le\, \exp(-c_2x^{\beta(m)}),
$$
where $W_m$ has distribution $F_m$.
\item $\beta(m)$ is strictly increasing, taking values in the interval $(1,2)$
and approaches $2$ as $m\to\infty$.
\item The sequence $(F_m)_{m\ge 1}$ approaches, as $m\to\infty$,
a centered Gaussian with variance $(2\pi)^{-1}$.
\end{enumerate}
\end{thm}

\begin{remark}
Note that parts \emph{b)} and \emph{c)} together state that the tails of $F_m$ are between those of an exponential
and a Gaussian. The fact that $\beta(m)$ approaches 2 is consistent with part \emph{d)}.
\end{remark}

One of the most interesting and studied sequence of Boolean functions corresponds to percolation crossings of a square. The rich structure of this particular example has inspired an extensive analysis; some parts of this recent development are presented in the book~\cite{garste}. We will state below a result for this example whose proof, unlike the proofs of all other results in this paper which are proved from first principles, will be based on some recent highly nontrivial developments in percolation and in so called near-critical percolation, due to Garban, Pete and Schramm~\cite{garpetsch13,garpetsch}, building on work by Kesten~\cite{kesten87}, see also~\cite{nolin08}.

To even begin this, we need to introduce a number of different concepts. However, we will 
be very brief and refer to~\cite{werner09} and~\cite{garste12} for background and explanation
of terms which are not clear. We consider percolation on the hexagonal lattice embedded 
into $\R^2$. Given $n$, we will consider the set of hexagons contained inside of 
$[0,n]\times [0,n]$, denoted by $B_n$, and we will think of these hexagons as indexing 
our underlying i.i.d.\ random variables of which we will then have approximately $n^2$.
We let $f_n$ be the indicator function of the event that there is a path of hexagons 
from the left side of this box to the right side all of whose values are 1. It is 
well known that there is a threshold at $p=1/2$ which is the critical value for percolation 
on the (full) hexagonal lattice. For critical percolation on the hexagonal lattice, we 
let $\alpha_4(R)$ be the probability that there are four paths of alternating value from a 
neighbor of the origin 0 to distance $R$ away; this event is usually called the 
{\it four-arm} event.
See Figure~\ref{f.armsevents} (in the appendix) for a realization of this event (where 1 is replaced by black and 0 is replaced by white).
Using Schramm-Loewner evolution and conformal invariance, it was proved by Smirnov and Werner~\cite{smiwer01} that
\[
\alpha_4(R) = R^{-\frac 5 4 +o(1) }\, \mbox{ as } R\to\infty.
\]

%\begin{figure}[htbp]
%\begin{center}
%\includegraphics[width=0.4\textwidth]{fig-quatrebras}
%\end{center}
%\caption{A realization of the {\it four-arm} event.}
%\label{f.armsevents}
%\end{figure}

A little bit of thought shows that if we have a hexagon $H$ in $B_n$, not too close to the 
boundary, which is pivotal for this crossing event, then the four-arm event to distance approximately $n$ 
centered at $H$ occurs. From here, it is possible to argue (see~\cite{garste12}) that 
the expected number of pivotal hexagons for $f_n$ is, up to constants, $n^2\alpha_4(n)$. 
This suggests what the proper scaling of $T_n$ should be and this turns out to be correct. 
The following result will be proved in an appendix to this paper authored by G\'abor Pete.

\begin{thm}\label{thm:Perc}
Consider percolation crossings of an $n\times n$-square of the hexagonal lattice.
\begin{enumerate}[label=\itshape\alph*),topsep=3pt,partopsep=1pt,itemsep=3pt,parsep=1pt]
\item Then $n^2\alpha_4(n)(T_n-\tfrac{1}{2})$ converges in distribution, as $n$ tends to infinity, to a random variable whose distribution $F$ is symmetric, absolutely continuous and fully supported on $\R$.
\item There exist constants $c_1$ and $c_2$ in $(0,\infty)$ so that for all $x\ge1$
$$
\exp(-c_1x^{4/3})\,\le\,\Pr(W\ge x)\,\le\,\exp(-c_2x^{4/3}),
$$
where $W$ has distribution $F$.
\end{enumerate}
\end{thm}

\begin{remark}
The fact that the limit in part~\emph{a)} of this theorem exists follows from recent results due to Garban, Pete and Schramm~\cite[Theorem~1.5 and Proposition~9.6]{garpetsch}, as stated already in a previous version of this paper. However, the precise rate of decay, stated in part~\emph{b)}, of the tails of the limiting distribution was not known to us at the time, and only later found by the appendix author.
\end{remark}

\subsection{Limiting behavior for general Boolean functions}

When one considers the question about the limiting distribution of $T_n$ for a given
sequence of Boolean functions, it is natural to ask
which distributions on $\R$ can arise as normalized limits of such a sequence of $T_n$. The next result,
proved in Section \ref{s.DistLimits}, says that they all do.
We remind the reader that a function $f:\{0,1\}^n\to\{0,1\}$ is called transitive if it is invariant with respect to a transitive group of permutations of $[n]$.

Part~\emph{b)} of the following theorem has been obtained jointly with Anders Martinsson.

\begin{thm}\label{thm:newLIMITS}
Let $\mu$ denote any probability measure on $\R$.
\begin{enumerate}[label=\itshape\alph*),topsep=3pt,partopsep=1pt,itemsep=3pt,parsep=1pt]
\item For any sequence $(a_n)_{n\ge1}$ satisfying $1\ll a_n\ll\sqrt{n}$,
there exists a sequence $(f_n)_{n\ge 1}$ of monotone functions 
$f_n:\{0,1\}^n\to\{0,1\}$ for which $a_n(T_n-\frac{1}{2})$ approaches $\mu$ in distribution.
\item For any sequence $(a_n)_{n\ge1}$ satisfying $\log n\ll a_n\ll\sqrt{n}$,
there exists a sequence $(f_n)_{n\ge 1}$ of monotone and transitive functions 
$f_n:\{0,1\}^n\to\{0,1\}$ for which $a_n(T_n-\frac{1}{2})$ approaches $\mu$ in distribution.
%\item For any sequences $(\ell_n)_{n\ge1}$ and $(a_n)_{n\ge1}$ satisfying $\ell_n\gg n^2$ and
%$\sqrt{\ell_n\log n}\ll a_n\ll\sqrt{n\ell_n}$, there exists a sequence $(f_n)_{n\ge 1}$
%of monotone and transitive functions $f_n:\{0,1\}^{n\lfloor\ell_n\rfloor}\to\{0,1\}$
%for which $a_n(T_n-\frac{1}{2})$ approaches $\mu$ in distribution.
\end{enumerate}
\end{thm}

\begin{remark}
By modifying the construction leading to part~\emph{b)} one may obtain, for any sequence $(a_n)_{n\ge1}$ satisfying $(\log n)^2\ll a_n\ll n$, a sequence $(f_n)_{n\ge1}$ of monotone graph properties, defined on $n$ vertices and ${n \choose 2}$ edges, for which $a_n(T_n-\frac{1}{2})$ approaches $\mu$ in distribution.
Moreover, the centralizing coefficient of Theorem~\ref{thm:newLIMITS} could in greater generality be replaced by any sequence $(b_n)_{n\ge1}$ bounded away from 0 and 1, although this may be of less interest.
We further mention that Rossignol~\cite{rossignol07} has previously showed that for any sufficiently smooth sequence $(a_n)_{n\ge1}$ satisfying $\log n\le a_n\le\sqrt{n}$ there exists an increasing sequence $(N(n))_{n\ge1}$ and monotone and transitive Boolean functions $(f_{N(n)})_{n\ge1}$ with a threshold at $1/2$ of width $1/a_n$.
\end{remark}

There are some conditions that the scaling coefficients $(a_n)_{n\ge1}$ have to meet in order to obtain a non-degenerate limit. Further restrictions apply in order not to impose properties on the limiting distribution. These facts are described in the following proposition. While some of these facts are well known to many we present them here for convenience to the reader, and provide a proof in Section~\ref{s.DistLimits}. Together they show that Theorem~\ref{thm:newLIMITS} is in fact sharp.

%We can never, in Theorem~\ref{thm:newLIMITS}, choose the sequence of scaling coefficients to be growing faster than $\sqrt{n}$, as the total influence for monotone Boolean functions on $n$ variables, for $p$ bounded away from 0 and 1, is of order at most $\sqrt{n}$. 
%For transitive functions the work of KKL~\cite{kahkallin88} and its extensions give a lower bound on the total influence of order $\log n$, yielding a matching lower bound on the scaling coefficients. 
%Moreover, it turns out that scaling is essential as not all distributions can arise without scaling. For transitive functions not all distributions may appear as limits when $a_n$ grows at rate $\log n$. In neither case may discontinuities appear in the limiting distribution when $a_n$ grows at rate $\sqrt{n}$.

%While some of these observations are well known to many, we summarize them here, for convenience of the reader, with a proof in Section~\ref{s.DistLimits}. Together they show that Theorem~\ref{thm:newLIMITS} is in fact sharp.

\begin{prop}\label{prop:newNoLIMITS}
Assume that $a_n(T_n-b_n)$ converges, as $n$ tends to infinity, to some non-degenerate probability measure $\mu$, and that $(b_n)_{n\ge1}$ is bounded away from $0$ and $1$. Then:
\begin{enumerate}[label=\itshape\alph*),topsep=3pt,partopsep=1pt,itemsep=3pt,parsep=1pt]
\item The sequence $(a_n/\sqrt{n})_{n\ge1}$ is bounded from above.
\item If the functions $(f_n)_{n\ge1}$ are transitive, then $(a_n/\log n)_{n\ge1}$ is bounded away from $0$.
\item If $(a_n)_{n\ge1}$ is bounded from above, then $\mu$ is necessarily fully supported on a (finite) interval.
\item If the functions $(f_n)_{n\ge1}$ are transitive and $(a_n/\log n)_{n\ge1}$ is bounded from above, then $\mu$ is necessarily fully supported on a (possibly infinite) interval.
\item If $(a_n/\sqrt{n})_{n\ge1}$ is bounded away from $0$, then $\mu$ is necessarily absolutely continuous.
\end{enumerate}
\end{prop}

\begin{remark}
Part~\emph{e)} of the above proposition was pointed out to us by Anders Martinsson. We also mention that the statement in the previous remark regarding graph properties is also essentially best possible, due to the work of Bourgain and Kalai~\cite{boukal97}.
\end{remark}

Interestingly, there are sequences of nondegenerate random variables $(X_n)_{n\ge 1}$
which are not {\em renormalizable} in the sense that for no subsequence $(X_{n_k})_{k\ge 1}$
are there normalizing constants $(a_k)_{k\ge1}$ and $(b_k)_{k\ge1}$ with the $a_k$'s 
nonnegative so that $a_{k}(X_{n_k}-b_k)$ converges, as $k\to\infty$, to a nondegenerate limiting 
distribution. A typical example of such a sequence is given by $X_n=e^{nZ}$ where $Z$ is a standard 
normal random variable. The vague idea is that when we try to scale down to keep mass from 
going to infinity, then the result will be that all the mass is accumulating at 0. 
Another example, which we will exploit, is when $X_n$ is uniformly distributed on 
$\{\pm 2^k: k=1,2,\ldots,n\}$.

The following proposition
shows that one cannot necessarily
extract a subsequence of $(T_n)_{n\ge1}$ which after normalization converges to a nondegenerate limit.

\begin{prop}\label{prop:NoSubseq}
There exists a (nondegenerate) sequence $(f_n)_{n\ge 1}$ of monotone (and transitive)
Boolean functions $f_n:\{0,1\}^n\to\{0,1\}$ so that no subsequence of $(T_n)_{n\ge1}$ can be renormalized to have a 
nondegenerate limiting distribution.
\end{prop}

\begin{remark}
The proof of this result, provided in Section~\ref{s.DistLimits}, will be based on Theorem~\ref{thm:newLIMITS}. We may therefore assume that the functions in the obtained sequence are transitive if we so wish.
\end{remark}

We end this introductory section with an open question.
As the reader may notice, all examples we have worked out yield a limiting distribution with exponentially or super-exponentially decaying tails. Although our Theorem~\ref{thm:newLIMITS} certainly shows that there are (sequences of) monotone Boolean functions giving rise to limiting distributions with heavier tails, the examples we construct are not very natural. Are there any ``natural'' examples whose limiting distributions present sub-exponential tails? 
%(A setting in which a limiting distribution with sub-exponential tail \emph{may} appear is discussed in the appendix.)

\section{Some elementary examples}\label{s.majority}

\subsection{Majority and the standard normal}

Majority is an example providing nontrivial, although classical, scaling behavior. The majority function 
on $n$ bits is defined to output the value 1 if there are at least $n/2$ bits with the value~1. 
More generality, we consider biased majority, which is the function with output 1 if and only if there are 
at least $pn$ bits valued 1, where $p\in(0,1)$ is a fixed parameter. The correct scaling factor will be 
of order $\sqrt{n}$ and the limit will be Gaussian, as stated in the following proposition.

\begin{prop}
For every $p\in(0,1)$ we have for the $p$-biased majority function on $n$ bits that $\sqrt{\frac{n}{p(1-p)}}(T_n-p)$ converges in distribution to a standard normal.
\end{prop}

Note that the multiplicative scaling is of order $\sqrt{n}$ and coincides with the order of
the total influence at the relevant parameter $p$; hence it is consistent with the heuristic described above.

\begin{proof}
Let $x\in\R$, $a_n=\sqrt{n/[p(1-p)]}$ and $p_n=p+x/a_n$. For large $n$ we have $p_n\in [0,1]$, and
$$
\Pr\big(a_n(T_n-p)\le x\big)\,=\,\Pr(T_n\le p_n)\,=\,\Pr_{p_n}\bigg(\sum_{i=1}^n\omega_i\ge np\bigg).
$$
We of course have a sum of $n$ Bernoulli variables with success probabilities $p_n$.

A consequence of the Lindeberg-Feller central limit theorem (see e.g.\ \cite[Theorem~3.4.5]{durrett10}) 
is that if $\{X_{i,n}:1\le i\le n,n\ge1\}$ is a family of bounded random variables, such that for each 
$n$, $\{X_{i,n}:1\le i\le n\}$ are i.i.d.\ with zero mean and variance that tends to 1 as $n$ increases, 
then $\sum_{i=1}^nX_{i,n}/\sqrt{n}$ converges in distribution to a standard normal.

Since $\Var_{p_n}(\omega_i)=p_n(1-p_n)$, which tends to $p(1-p)$, and $(np-np_n)/\sqrt{np(1-p)}=-x$, the 
above consequence of the Lindeberg-Feller theorem implies that, as $n\to\infty$,
$$
\Pr_{p_n}\bigg(\sum_{i=1}^n\omega_i\ge np\bigg)\;=\;\Pr_{p_n}\bigg(\sum_{i=1}^n\frac{\omega_i-p_n}{\sqrt{np(1-p)}}\ge\frac{np-np_n}{\sqrt{np(1-p)}}\bigg)\;\to\;\Phi(x),
$$
the distribution function of a standard normal distribution.
\end{proof}

\subsection{Tribes, Gumbel and coupon collectors}

The tribes function on $n$ bits is defined as follows. Given $\ell_n$, partition $[n]$ into 
$\lfloor n/\ell_n\rfloor$ sets (`tribes') of length $\ell_n$ (plus some residual bits). 
Then $f_n(\omega)=1$ if and only if $\omega$ is all 1's for at least one tribe. 
The {\em correct} choice for $\ell_n$, in order for the distribution of to be nondegenerate for the uniform measure, is of order $\log_2n-\log_2\log_2n$.

\begin{prop}\label{prop:Tribes}
Consider tribes with $\ell_n=\lfloor\log_2n-\log_2\log_2n\rfloor$, set $\alpha_n=(\log_2n-\log_2\log_2n)/\ell_n$. Then for all $x\in\R$ we have
$$
\lim_{n\to\infty}\Pr\Big(2(\log_2n)\big(T_n-\big(\tfrac{1}{2}\big)^{\alpha_n}\big)\le x\Big)\to 1-\exp(-e^{x}).
$$
\end{prop}

Note that the multiplicative scaling is of order $\log_2 n$, which can be checked to be the order of
the total influence at the relevant parameter $1/2$; again, this is consistent with the heuristic described in
the introduction. Note also that the upper tail of this limiting distribution 
decays super-exponentially, whereas the lower tail just decays exponentially.

\begin{proof}
Note that the $\lfloor n/\ell_n\rfloor$ tribes of length $\ell_n$ corresponds to the 1-witnesses for the tribes function. Let $X_n$ denote the number of tribes (1-witnesses) for which $\omega$ is all 1. For $\omega\sim\Pr_p$, we see that $X_n$ is binomially distributed with parameters $\lfloor n/\ell_n\rfloor$ and $p^{\ell_n}$. Clearly, $\lfloor n/\ell_n\rfloor\to\infty$ as $n\to\infty$. Also, for all $x\in\R$ we have
$$
\lfloor n/\ell_n\rfloor\left(\big(\tfrac{1}{2}\big)^{\alpha_n}+\frac{x}{2\log_2n}\right)^{\ell_n}\,=\,\lfloor n/\ell_n\rfloor\big(\tfrac{1}{2}\big)^{\alpha_n\ell_n}\left(1+\frac{x}{2^{1-\alpha_n}\log_2n}\right)^{\ell_n}\,\to\, e^x
$$
as $n\to\infty$.
Given $x\in\R$ and letting $p_n=(\tfrac{1}{2})^{\alpha_n}+\frac{x}{2}\log_2n$, we therefore have,
by the Poisson convergence theorem (see e.g.\ \cite[Theorem~3.6.1]{durrett10}), that
for $\omega\sim\Pr_{p_n}$, $X_n(\omega)$ converges in law to a Poisson distribution with parameter $e^x$.  
Since for each $n$,
$$
\Pr\left(T_n\le\big(\tfrac{1}{2}\big)^{\alpha_n}+\frac{x}{2\log_2n}\right)\,=\,\Pr_{p_n}(X_n\ge1),
$$
we thus conclude that
$$
\lim_{n\to\infty}\Pr\left(T_n\le\big(\tfrac{1}{2}\big)^{\alpha_n}+\frac{x}{2\log_2n}\right)\,=
\,1-\exp(-e^{x}),
$$
as we needed to show.
\end{proof}

\begin{remark}
The unfortunate term $\alpha_n$ arises due to the fact that
$\log_2n-\log_2\log_2n$ is not an integer. A related fact is that if $p=1/2$, then 
the number of tribes which are identically 1 has all Poisson distributions with parameter in $[1,2]$ as
subsequential limits. If we were to restrict outselves to $n$'s of the form $2^{2^k}$, then 
$\log_2n-\log_2\log_2n$ would be an integer and we would have the simpler form that
$$
\lim_{n\to\infty}\Pr\Big(2(\log_2n)\big(T_n-\tfrac{1}{2}\big)\le x\Big)\to 1-\exp(-e^x)
$$
along this thin subsequence of $n$.
\end{remark}

The reader might recognize the limiting distribution obtained in Proposition~\ref{prop:Tribes}. 
In general if $X$ has distribution $F(x)$, then $-X$ has distribution $1-F(-x)$. If $Y$ is distributed
according to the above limiting distribution, then $-Y$ has distribution
$\exp(-e^{-x})$ which is known as the standard Gumbel distribution.
This distribution often arises in extreme value theory and in
particular is the limiting distribution after proper normalization of
\begin{enumerate*}[label=\it\alph*)]
\item the maximum of $n$ independent unit exponential random variables (where
one subtracts $\log n$ but uses no scaling factor to normalize); and
\item the number of picks needed to collect $n$ coupons when each pick is uniform
(where one subtracts $n\log n$ and divides by $n$ to normalize).
\end{enumerate*}
Heuristically, the reason that one gets the same limiting distribution in these two models
is that in the latter case, we have the maximum of $n$ weakly dependent geometric random variables
with parameters $1/n$. When dividing by $n$ (which explains the difference of a factor of $n$ in the two normalizations),
the geometric random variables become unit exponentials in the limit.

While we will not give an alternative proof of Proposition \ref{prop:Tribes} based on these ideas,
we want to {\em explain} why we obtained the limiting distribution there that we did.

Given a Boolean function $f$, define its reversal $\hat{f}$ by $\hat{f}(\omega)=1-f(1-\omega)$ and observe 
that $\hat{f}$ is also a monotone Boolean function. One immediately checks that 
$T(\hat{f})$ and $1-T(f)$ have the same distribution. If $f_n$ is our tribes function, then this
distributional relationship and Proposition \ref{prop:Tribes} easily yields that
\be\label{eq:LimitBackwards}
2(\log_2n)\big(T(\hat{f_n})-1+\big(\tfrac{1}{2}\big)^{\alpha_n}\big)
\ee
converges to the standard Gumbel distribution.
We now give a heuristic for this. Clearly, $\hat{f}_n$ is the function which is 0 if and only if there
is a tribe which is all 0's. (The tribes are 0-witnesses for $\hat{f}_n$.) One easily checks that 
$T(\hat{f_n})$ is the smallest $p$ such that each tribe has a 1 in it with respect to $\eta_p$; compare with~\eqref{eq:WitnessCharacterization}. The distribution of the time
at which a given tribe gets its first 1 is equal to the distribution of the minimum of $\ell_n$ uniform
random variables. The minimum of $k$ uniform random variables after multiplying by $k$ 
converges to a unit exponential. Therefore, since different tribes are disjoint 
(and hence their corresponding uniform random
variables are independent) and have size $\ell_n$, it follows that $\ell_n T(\hat{f_n})$ 
is approximately the maximum of $\lfloor n/\ell_n\rfloor$ unit exponential random variables.
 Therefore one {\em should} have that
$$
\ell_n T(\hat{f_n})-\log (\lfloor n/\ell_n\rfloor)\,
=\,\ell_n \left(T(\hat{f_n})-\frac{\log (\lfloor n/\ell_n\rfloor)}{\ell_n}\right)
$$ 
converges to the Gumbel distribution. This is certainly close to 
(\ref{eq:LimitBackwards}) and heuristically explains the {\em reverse}
Gumbel distributional limit.

\begin{remark}
The so-called circular tribes function is a more symmetric version of tribes and perhaps more natural.
It is defined as follows. We place the $n$ bits in a circle and define $f_n(\omega)$ to be 1 if  
$\omega$ contains an interval of 1's of length $\lfloor\log_2n\rfloor$. One can prove in a similar manner
that the corresponding sequence $T_n$ also has the reverse Gumbel distribution as a limit. 
The situation is however slightly different than for tribes since the number of such intervals 
containing all 1's is no longer Poisson but rather compound Poisson, where the summands are mean 
2 geometric random variables. 
\end{remark}

\subsection{Random graph properties}

In this subsection we cover a few monotone functions related to random graphs. We remind the reader 
that a random graph on $n$ vertices is obtained by declaring each of the possible ${n\choose 2}$ 
edges open with probability $p\in(0,1)$. Equivalently, this amounts to determining an element 
$\omega\in\{0,1\}^{n\choose 2}$ according to $\Pr_p$. We first discuss two functions whose 
critical values occur near 0. The proof of the following proposition is very straightforward 
(when using well known results) and hence we only sketch the proof.

\begin{prop}
\begin{enumerate}[label=\itshape\alph*),topsep=3pt,partopsep=1pt,itemsep=3pt,parsep=1pt]
\item[]
\item Let $f_n$ be the function corresponding to containing a triangle in a graph with $n$ vertices.
Then, for all $x\ge 0$, we have that
$$
\lim_{n\to\infty}\Pr\big(nT_n\le x\big)\,=\, 1-\exp(-x^3\!/6).
$$
\item Let $f_n$ be the function corresponding to a graph with $n$ vertices being connected.
Then, for all $x\in \R$, we have that
$$
\lim_{n\to\infty}\Pr\big(nT_n-\log n\le x\big)\,=\, \exp(-e^{-x}).
$$
\end{enumerate}
\end{prop}

The multiplicative scaling $n$ can in both cases be checked to be the order of the total 
influence at the relevant parameter.

\begin{proof}
\emph{a)} It is well known (see e.g.\ \cite[Theorem~4.1]{bollobas01}) that if $p=x/n$, then the number
of triangles contained in the random graph converges to a Poisson distribution with parameter
$x^3/6$. The result follows immediately using~\eqref{eq:identity}.

\emph{b)} It is well known (see e.g.\ \cite[Theorem~7.3]{bollobas01}) that for any $x\in \R$,
if $p=(\log n +x)/n$, the probability that the random graph is connected approaches
$\exp(-e^{-x})$. The result follows.
\end{proof}

\begin{remark}
We see that in part~\emph{a)} of the above proposition the threshold is coarse and the support of the limiting distribution bounded to the left, while in part~\emph{b)} the threshold is sharp and the support unbounded to the left. This in an instance of the general phenomenon that only sharp thresholds may give rise to limiting distributions supported on the whole line.
More precisely, assume that $(b_n)_{n\ge1}$ is bounded away from 1 and that $a_n(T_n-b_n)$ converges, as $n$ tends to infinity, to some probability measure $\mu$. Then a coarse threshold is characterized by the sequence $(a_nb_n)_{n\ge1}$ being bounded above, and if $c$ is an upper bound on this sequence, then the support of $\mu$ is contained in $[-c,\infty)$.
%Similarly, if $(b_n)_{n\ge1}$ is bounded away from zero and $(a_n[1-b_n])_{n\ge1}$ bounded above by $c$, then the support of $\mu$ is contained in $(-\infty,c]$.
\end{remark}

A \emph{clique} is a maximal complete subgraph of a graph. At a given parameter $p\in(0,1)$, the 
expected number of complete subgraphs of size $\ell$ of a random graph on $n$ vertices falls abruptly 
from being very large to being very small, as $\ell$ increases. As a consequence, the maximal clique 
size of a random graph is highly concentrated, with high probability equal to either of two 
consecutive values $\ell_n-1$ or $\ell_n$, where $\ell_n=\ell_n(p)$. Using Stirling's approximation 
one sees that this sequence must satisfy $\ell_n\sim2\log_{1/p}n$. This is well known; 
see e.g.\ \cite[Chapter~4]{bollobas01}.

We will be interested in the function encoding the existence of a clique of size 
$\ell_n$. For most values of $n$ the maximal sized clique consists of $\ell_n$ vertices with 
probability close to 1. However, along certain subsequences this probability remains bounded away 
from 1. Instead of restricting to subsequences we may allow $p$ to vary, similar to the case of 
tribes. We simply state this result without proof since the argument follows more or less
the argument for tribes. One obtains the result by proving Poisson approximation for the 
number of complete graphs of a given size. While this is more involved than for tribes, it is proved in
\cite[Theorems~11.7 and~11.9]{bollobas01}.

\begin{prop}
Let $p\in(0,1)$ and $\ell_n=\ell_n(p)$ be the above mentioned sequence. 
Let $p_1,p_2,\ldots$ be any sequence bounded away from 0 and 1 such that the limit
$$
\lambda\,:=\,\lim_{n\to\infty}{n\choose\ell_n}p_n^{\ell_n\choose2}\text{ exists in }(0,\infty).
$$
Then, for the Boolean function encoding the existence of a complete graph of size $\ell_n$, we have
$$
\lim_{n\to\infty}\Pr\bigg(\frac{\ell_n^2}{2p_n}\big(T_n-p_n\big)\le x\bigg)\,=\,1-\exp(-\lambda e^x)\quad\text{for }x\in\R.
$$
\end{prop}

\section{Iterated majority}\label{s.IM}

In this section, we will analyze iterated majority and prove Theorem~\ref{thm:m-majority}. In order to understand the asymptotic behavior of iterated majority, one is led to study its recursive structure. 
The limiting distribution will be described through the iterates of some function $g:[0,1]\to[0,1]$, 
and the appropriate scaling will be determined by the derivative of $g$ at $\frac{1}{2}$.

We begin by describing what the limiting distribution $F_m$ will be. Define $g:[0,1]\to[0,1]$ as 
the probability at parameter $x\in[0,1]$ that the majority on $m$ bits equals 1. Formally, $g$ is given by
$$
g(x)\,=\sum_{k= (m+1)/2}^m\binom{m}{k}x^k(1-x)^{m-k}.
$$
Observe that $\gamma(m)$, which will be our scaling coefficient satisfies
\be\label{eq:gamma}
\gamma(m):=\,\frac{m}{m-1}\,\gamma(m-2).
\ee
It is clear that $\gamma(m)$ is increasing in $m$, and Stirling's approximation says that 
$\gamma(m)\sim\sqrt{2m/\pi}$ as $m$ tends to infinity. 
It turns out that the total influence for $f_n$, i.e.\ iterated majority on $m^n$ bits (where $m$ is implicit), is $\gamma^n$ and
we will below see that $\gamma$ coincides with the derivative of $g$ at $\frac{1}{2}$.
The recursive structure of $\gamma(m)$ stated in~\eqref{eq:gamma}
easily yields that $\gamma(m)<\frac{m}{2}$, implying in turn that $\beta(m)>1$ for all $m\ge3$. 
Also, using $\gamma(m)\sim\sqrt{2m/\pi}$, we find that $\beta(m)\to2$ as $m\to\infty$.

It turns out to be convenient to consider the translate
$$
h(x)=g(\tfrac{1}{2}+x)-\tfrac{1}{2}
$$
of $g$. The scaling limit of iterated $m$-majority will be described in terms of the limit as $n\to\infty$ of $h^{(n)}(\alpha\gamma^{-n})$, where $\alpha\in\R$ and $h^{(n)}$ denotes the composition of $h$ with itself $n$ times.

We will break up the proof of the four parts of Theorem~\ref{thm:m-majority} into subsections.

\subsection{Proof of part \emph{a)}}

We begin with the following proposition which will be central for our analysis.

\begin{prop}\label{prop:L}
For every odd integer $m\ge3$ and $\alpha\in\R$, the limit $L(\alpha):=\lim_{n\to\infty}h^{(n)}\big(\alpha\gamma^{-n}\big)$ exists and the 
resulting function $L:\R\to(-\frac{1}{2},\frac{1}{2})$ is odd, onto, 1-Lipschitz 
continuous, strictly increasing and continuously differentiable.
\end{prop}

We first need the following lemma.

\begin{lma}\label{lemma:L}
The function $h:[-\frac{1}{2},\frac{1}{2}]\to[-\frac{1}{2},\frac{1}{2}]$ is odd, onto, strictly increasing, strictly convex on $[-\frac{1}{2},0]$ and strictly concave on $[0,\frac{1}{2}]$. In particular $h'(x)\le h'(0)=\gamma(m)$ for $x\in[-\frac{1}{2},\frac{1}{2}]$.
\end{lma}

\begin{proof}
It suffices to demonstrate the corresponding characteristics for $g$. From the interpretation of $g$ as a probability, it is clear that $g$ is strictly increasing, maps $0$, $\frac{1}{2}$ and $1$ to themselves, and that
$$
g(x)\,=\,\Pr\big(\Bin(m,x)\ge m/2\big)\,=\,1-\Pr\big(\Bin(m,x)< m/2\big)\,=\,1-g(1-x).
$$
Thus $h$ is odd, strictly increasing, has fixed points at $-\frac{1}{2}$, $0$ and $\frac{1}{2}$, and is therefore also onto.

We know that $g$ is differentiable and we aim to determine its derivative. Note that
$$
g'(1)\,=\,g'(0)\,=\,\lim_{x\to0}\sum_{k\ge m/2}\binom{m}{k}x^{k-1}(1-x)^{m-k}\,=\,0.
$$
Next, pick $\delta>0$ and let $\xi_1,\xi_2,\ldots,\xi_m$ be independent and $[0,1]$-uniformly distributed. Using the monotone coupling we find that
\bea
g\big(x+\delta(1-x)\big)-g(x)\,=\,\Pr\Big(\#\{\xi_i\le x\}<m/2,\#\{\xi_i\le x+\delta(1-x)\}\ge m/2\Big).
\eea
Conditioning on the number of $\xi_i$'s whose value is at most $x$ we arrive at
$$
\sum_{k\le m/2}\binom{m}{k}x^{k}(1-x)^{m-k}\,\Pr\Big(\#\{\xi_i\le x+\delta(1-x)\}\ge m/2\,\Big|\,\#\{\xi_i\le x\}=k\Big).
$$
The above conditional probability coincides with the probability that a binomial random variable with parameters $m-k$ and $\delta$ is at least $m/2-k$, and is thus independent of $x$. In addition,
$$
%\lim_{\delta\to0}
\delta^{-1}\,\Pr\big(\Bin(m-k,\delta)\ge m/2-k\big)\,=\sum_{\ell\ge m/2-k}\binom{m-k}{\ell}\delta^{\ell-1}(1-\delta)^{m-k-\ell},
$$
and sending $\delta$ to 0 leaves us with $m-k=(m+1)/2$ in case $k=(m-1)/2$, and $0$ for all smaller values of $k$. In conclusion, for $x\in(0,1)$,
\be\label{eq:gprime}
g'(x)\,=\,\lim_{\delta\to0}\frac{g\big(x+\delta(1-x)\big)-g(x)}{\delta(1-x)}\,=\,\frac{m+1}{2}\binom{m}{\frac{m-1}{2}}\big[x(1-x)\big]^{\frac{m-1}{2}}.
\ee
Differentiating once more gives
$$
g''(x)\,=\,\frac{m+1}{2}\frac{m-1}{2}\binom{m}{\frac{m-1}{2}}\big[x(1-x)\big]^{\frac{m-3}{2}}(1-2x).
$$

In conclusion, the derivative of $g$ is strictly positive on $(0,1)$, and the second derivative is strictly positive on $(0,\frac{1}{2})$ and strictly negative on $(\frac{1}{2},1)$. So $h$ possesses the claimed properties and $h'$ reaches its maximum at the origin, which is easily seen to equal $\gamma(m)$.
\end{proof}

The proof of Proposition~\ref{prop:L} will make repeated use of the properties of $h$ displayed in Lemma~\ref{lemma:L}. For instance, we note that $h$ cannot have any fixed points other than $-\frac{1}{2}$, $0$ and $\frac{1}{2}$.

\begin{proof}[Proof of Proposition~\ref{prop:L}]
Since $h(0)=0$ we also have $L(0)=0$, and since $h$ is odd the limit $L(\alpha)$, 
if it exists, has to be odd as well. In particular, it will be sufficient to consider $\alpha\ge0$ 
for the rest of this proof.

\emph{Existence.} Given $\alpha\ge0$, choose $n_0$ such that $\alpha\gamma^{-n}\le\frac{1}{2}$ for all $n\ge n_0$. Note that we may obtain $h^{(n)}(\alpha\gamma^{-n})$ from $\alpha\gamma^{-(n+1)}$ by first multiplying by $\gamma$, and then applying $h$ $n$ times. $h^{(n+1)}(\alpha\gamma^{-(n+1)})$ is similarly obtained from $\alpha\gamma^{-(n+1)}$ by first applying $h$ once, and then another $n$ times. Lemma~\ref{lemma:L} shows that the derivative of $h$ is bounded by $\gamma$. Hence $\gamma x\ge h(x)$ for all $x\in[0,1/2]$, and it follows that $h^{(n)}(\alpha\gamma^{-n})$ is decreasing in $n$ for $n\ge n_0$. Since the sequence is bounded below by $0$,
the limit $L(\alpha)$ necessarily exists for all $\alpha\ge0$.

\emph{1-Lipschitz Continuity.} Using again that $|h'|\le\gamma$, together with iterated use of the mean value theorem, we find for $\alpha,\alpha'\in\R$ that
\bea
\big|L(\alpha)-L(\alpha')\big|\,=\lim_{n\to\infty}\big|h^{(n)}(\alpha\gamma^{-n})-h^{(n)}(\alpha'\gamma^{-n})\big|\,\le\,\liminf_{n\to\infty}\gamma^n\big|\alpha\gamma^{-n}-\alpha'\gamma^{-n}\big|\,=\,|\alpha-\alpha'|,
\eea
where we also have used that $\alpha\gamma^{-n}$ and $\alpha'\gamma^{-n}$ are contained in $[-\frac{1}{2},\frac{1}{2}]$ for large $n$.
%and that $|h'|$ is bounded above by $\gamma$.

An observation that will be important for the rest of this proof is that, by continuity of $h$, for all $\alpha\in\R$
\be\label{eq:hL}
h\big(L(\alpha)\big)\,=\,\lim_{n\to\infty}h^{(n+1)}(\alpha\gamma^{-n})\,=\,L(\alpha\gamma).
\ee
Iterating this yields that
\be\label{eq:hLIterate}
L(\alpha)\,=\,h\big(L(\alpha\gamma^{-1})\big)\,=\,h^{(n)}\big(L(\alpha\gamma^{-n})\big).
\ee

\emph{Strict Monotonicity.} Note that weak monotonicity of course follows from $h$ being increasing. We will next aim to show that for $\alpha\ge\alpha'\ge0$ sufficiently small, we have
\be\label{eq:Llower}
L(\alpha)-L(\alpha')\,\ge\,(\alpha-\alpha')\prod_{k=1}^\infty\big(1-\alpha\big(\tfrac{3}{4}\big)^{k}\big).
\ee
Apart from showing that $L$ is strictly increasing in a neighborhood around the origin,~\eqref{eq:Llower} will be an important step in the proof of differentiability of $L$. Note that strict monotonicity of $L$ would follow for all $\alpha\in\R$ by~\eqref{eq:hLIterate} and~\eqref{eq:Llower}, since $h$ is strictly increasing.

We now deduce~\eqref{eq:Llower}. Using concavity of $h$ on $[0,\frac{1}{2}]$, we observe that $\frac{4}{3}\le\gamma-\frac{1}{6}\le h'(x)\le\gamma$ on some interval $[0,c]$, where $c=c(m)>0$. So, by the mean value theorem we conclude that
$$
\tfrac{4}{3}x\,\le\, h(x)\,\le\,\gamma x\quad\text{on }[0,c].
$$
Consequently, for all 
$\alpha\in[0,c]$ and $1\le k\le n$, we have $h^{(k)}(\alpha\gamma^{-n})\le\alpha$, and therefore
$$
h^{(k)}(\alpha\gamma^{-n})\,\le\,\tfrac{3}{4}\,h^{(k+1)}(\alpha\gamma^{-n})\,\le\,\big(\tfrac{3}{4}\big)^{n-k}\,h^{(n)}(\alpha\gamma^{-n})\,\le\,\alpha\big(\tfrac{3}{4}\big)^{n-k}.
$$
Now, for any $\alpha'\le\alpha$ in $[0,c]$ and given $n$, we obtain, from iterated use of the mean value theorem, 
the existence of constants $\{s_k^n\}_{1\le k\le n}$ with $s_k^n\in\big[h^{(k-1)}\big(\alpha'\gamma^{-n}\big),h^{(k-1)}\big(\alpha\gamma^{-n}\big)\big]$ and such that
\be \label{eq:FirstNeeded}
h^{(n)}\big(\alpha\gamma^{-n}\big)-h^{(n)}\big(\alpha'\gamma^{-n}\big)\,=\,\big(\alpha\gamma^{-n}-\alpha'\gamma^{-n}\big)\prod_{k=1}^nh'(s_k^n).
\ee
Since $h'$ is decreasing on $[0,\frac{1}{2}]$ and $h''(0)=0$, we have that $h'(x)$ is bounded below by $\gamma(1-x)$ on some, possibly smaller, interval $[0,c']$. As a consequence we obtain the lower bound on~\eqref{eq:FirstNeeded},
\be \label{eq:SecondNeeded}
\begin{aligned}
\big(\alpha\gamma^{-n}-\alpha'\gamma^{-n}\big)\prod_{k=1}^n\gamma\big(1-h^{(k-1)}(\alpha\gamma^{-n})\big)\;&\ge\;(\alpha-\alpha')\prod_{k=1}^n\big(1-\alpha\big(\tfrac{3}{4}\big)^{n-k+1}\big)\\
=\;(\alpha-\alpha')\prod_{k=1}^n\big(1-\alpha\big(\tfrac{3}{4}\big)^{k}\big)
\;&\ge\;(\alpha-\alpha')\prod_{k=1}^\infty\big(1-\alpha\big(\tfrac{3}{4}\big)^{k}\big).
\end{aligned}
\ee
Combining (\ref{eq:FirstNeeded}) and (\ref{eq:SecondNeeded}) and letting $n\to\infty$ yields~\eqref{eq:Llower} for every $\alpha'\le\alpha$ in $[0,c']$.

\emph{Continuous Differentiability.} Using~\eqref{eq:hLIterate} we have for any $\alpha\ge0$ and $\delta\in\R$ that
\begin{equation*}
\begin{aligned}
L(\alpha+\delta)-L(\alpha)\;&=\;h^{(n)}\big(L\big((\alpha+\delta)\gamma^{-n}\big)\big)-h^{(n)}\big(L(\alpha\gamma^{-n})\big)\\
&=\;\Big[L\big((\alpha+\delta)\gamma^{-n}\big)-L(\alpha\gamma^{-n})\Big]\prod_{k=1}^nh'\big(h^{(k-1)}\big(L(\alpha_k\gamma^{-n})\big)\big)
\end{aligned}
\end{equation*}
and where we in the last step have used the mean value theorem iteratively; the $\alpha_k$'s are bounded between $\alpha$ and $\alpha+\delta$. By continuity and monotonicity of $h$ and $L$, these $\alpha_k$'s exist. Using $h^{(k-1)}\big(L(\alpha)\big)=L(\alpha\gamma^{k-1})$, and reindexing the terms of the product, we arrive at
\be\label{eq:diffequal}
\frac{L(\alpha+\delta)-L(\alpha)}{\delta}\,=\,\frac{L\big((\alpha+\delta)\gamma^{-n}\big)-L(\alpha\gamma^{-n})}{\delta\gamma^{-n}}\prod_{k=1}^n\gamma^{-1}h'\big(L(\alpha_{n-k+1}\gamma^{-k})\big).
\ee

%The above identity holds for all $n$.
We now want to take limits. First, since $L$ is 1-Lipschitz continuous,
$$
\limsup_{\delta\to0}\frac{L(\alpha+\delta)-L(\alpha)}{\delta}\,\le\,\prod_{k=1}^n\gamma^{-1}h'\big(L(\alpha\gamma^{-k})\big),
$$
which is decreasing in $n$. Second, we note that the infinite product in~\eqref{eq:Llower} tends to 1 as $\alpha\to0$. Applying this to the first term in~\eqref{eq:diffequal}, we conclude that for every $\eps>0$, if $n$ is sufficiently large, then
%\eqref{eq:Llower} and~\eqref{eq:diffequal} together give that
$$
\liminf_{\delta\to0}\frac{L(\alpha+\delta)-L(\alpha)}{\delta}\,\ge\,(1-\eps)\prod_{k=1}^n\gamma^{-1}h'\big(L(\alpha\gamma^{-k})\big).
$$
Sending $n$ to infinity, and then $\eps$ to zero, we conclude that the inferior and superior limits coincide and that
\be\label{eq:Lprime}
L'(\alpha)\,=\,\lim_{\delta\to0}\frac{L(\alpha+\delta)-L(\alpha)}{\delta}\,=\,\prod_{k=1}^\infty\gamma^{-1}h'\big(L(\alpha\gamma^{-k})\big).
\ee
Since $h'\le\gamma$ the limit is finite, and since $L(\alpha\gamma^{-k})\le\alpha\gamma^{-k}$ and $h'(x)\ge\gamma(1-x)$ for small $x\ge0$, the limit is strictly positive for all $\alpha\in\R$. This, again, shows that $L$ is strictly monotone on $\R$.

We need to show that $L'$ is continuous, and note, based on~\eqref{eq:Lprime}, that $L'$ is decreasing on $[0,\infty)$ since $L$ is increasing. Since also $L'>0$ on $\R$ it follows that
$$
\prod_{k=\ell}^\infty\gamma^{-1}h'\big(L(\alpha\gamma^{-k})\big)\to1\quad\text{as }\ell\to\infty
$$
uniformly on compact sets. Thus, for every $\eps>0$
$$
\left|\lim_{x\to\alpha}L'(x)-\prod_{k=1}^\ell\gamma^{-1}h'\big(L(\alpha\gamma^{-k})\big)\right|<\eps
$$
for large enough $\ell$, showing that $\lim_{x\to\alpha}L'(x)=L'(\alpha)$.

\emph{Surjectivity.} Since $L(0)=0$ and $L$ is continuous, it remains to show that $L(\alpha)\to\tfrac{1}{2}$ as $\alpha\to\infty$. For any $k\in\N$ we have from~\eqref{eq:hLIterate} that $L(\alpha\gamma^k)=h^{(k)}\big(L(\alpha)\big)$. Since $L(\alpha)>0$ for $\alpha>0$, the properties of $h$ imply that
$h^{(k)}\big(L(\alpha)\big)\to\frac{1}{2}$ as $k\to\infty$. 
Together with the proven continuity of $L$ and its weak monotonicity,
this shows that $L$ maps $[0,\infty)$ onto $[0,\frac{1}{2})$. 
\end{proof}

We now use the above to analyze the asymptotics of $T_n$ for iterated $m$-majority on $m^n$ bits to prove part~\emph{a)} of Theorem~\ref{thm:m-majority}.
With $x\in\R$ fixed, the goal will be to relate, for large $n$, the probability $\Pr(T_n\le p_n)$, 
where $p_n=\frac{1}{2}+x\gamma^{-n}$, with the function $L:\R\to[-\frac{1}{2},\frac{1}{2}]$ 
defined in Proposition~\ref{prop:L}.

Given $p\in(0,1)$ and $n\in\N$, let
$$
q_n(p)\,:=\,\Pr_{p}\big(f_n(\omega)=1\big).
$$
Using the iterated majority structure, it is easy to express $q_n(p)$ in terms of $q_{n-1}(p)$.
Specifically, $q_n(p)$ is the probability that a binomial random 
variable with parameters $m$ and $q_{n-1}(p)$ is at least $m/2$. That is,
$$
q_n(p)\,=\,\Pr\big(\textup{Bin}(m,q_{n-1}(p))\ge m/2\big)\,=\,g\big(q_{n-1}(p)\big),
$$
where $g$ is defined as above. Of course, the base case $q_0(p)$ equals the probability of the majority function on one bit being $1$, which equals $p$. Iterating the above procedure we get $q_n(p)=g^{(n)}(p)$. Replacing $p$ by $p_n=\frac{1}{2}+x\gamma^{-n}$ leaves
$$
\Pr(T_n\le p_n)\,=\,\Pr_{p_n}\big(f_n(\omega)=1\big)\,=\,q_n(p_n)\,=\,g^{(n)}\big(\tfrac{1}{2}+x\gamma^{-n}\big)\,=\,\tfrac{1}{2}+h^{(n)}(x\gamma^{-n}),
$$
which, according to Proposition~\ref{prop:L}, converges as $n\to\infty$ to $\frac{1}{2}+L(x)$.
%$$
%\lim_{k\to\infty}\Pr(T_n\le p_k)\,=\,\tfrac{1}{2}+L(\alpha).
%$$

That the limiting distribution is absolutely continuous and fully supported on $\R$ is a consequence 
of the properties of $L$ established in Proposition~\ref{prop:L}. The final sentence of part~\emph{a)} follows
from parts~\emph{b)} and~\emph{c)}.

\begin{remark}
Observe that the first equality in (\ref{eq:hLIterate}) has a very nice interpretation.
It says that the two homeomorphisms $x\mapsto \gamma x$ from $\R$ to itself and
$x\mapsto h(x)$ from $[-1/2,1/2]$ to itself are conjugate and that $L$ provides a conjugation between
these homeomorphisms (showing that they are conjugate). We know that $L$ is also a
homeomorphism. Since the two conjugate mappings are analytic, one might guess that one can show
that $L$ has good properties (such as analyticity or being continuously differentiable) by virtue of
the fact it is a conjugacy between these systems. This is unfortunately not true. One can construct
self-conjugacies $f$ of $x\mapsto \gamma x$ which, while being homeomorphisms,
are not continuously differentiable. Being a self-conjugacy amounts to saying that
$f(\gamma x)=\gamma f(x)$ and so in particular we are saying that
$f(\gamma x)=\gamma f(x)$ does not imply that $f$ is linear even for homeomorphisms.
Since $L\circ f$ would also be a conjugacy between the two systems, one cannot conclude good properties
of $L$ using only the fact that it is a conjugacy.
\end{remark}

\subsection{Proof of part \emph{b)}}

We first need the following lemma.

\begin{lma}\label{lma:hDetailed}
For every odd integer $m\ge3$ there exists $\eps_0=\eps_0(m)>0$ such that for all $\ell\ge 1$ and $\eps \in (0,\eps_0)$,
$$
%2^{(\frac{m+1}{2})^{\ell-1}}
\eps^{(\frac{m+1}{2})^\ell}\,
\le\, \left|h^{(\ell)}\big(\tfrac{1}{2}-\eps\big)-\tfrac{1}{2}\right|\,\le\, (9\eps)^{(\frac{m+1}{2})^\ell}.
$$
\end{lma}

\begin{proof}
First note that
$$
\tfrac{1}{2}-h\big(\tfrac{1}{2}-\eps\big)\,=\,1-g(1-\eps)\,=\,g(\eps),
$$
which after iteration leaves us with $\frac{1}{2}-h^{(\ell)}(\frac{1}{2}-\eps)=g^{(\ell)}(\eps)$.
In addition, each term in $g$ with non-zero coefficient has degree at least $(m+1)/2$. Thus,
$$
g(x)=\binom{m}{\frac{m+1}{2}}x^\frac{m+1}{2}\big(1+r(x)\big)
$$
for some polynomial $r(x)\to0$ as $x\to0$. Since
$$
\binom{m}{\frac{m+1}{2}}\,=\,\binom{m}{\frac{m-1}{2}}\,\le\,\bigg(\frac{2em}{m-1}\bigg)^\frac{m-1}{2}\le\,9^{\frac{m-1}{2}},
$$
we see that $g(\eps)$ lies in $\big[\eps^\frac{m+1}{2},(9\eps)^\frac{m+1}{2}/9\big]$ for all sufficiently small $\eps$. Using the fact that $g$ is increasing and $g(x)\le x$ on $(0,\frac{1}{2})$, it then follows by two inductions
%(one for the lower bound and one for the upper bound)
that $\eps^{(\frac{m+1}{2})^\ell}$ and $(9\eps)^{(\frac{m+1}{2})^\ell}/9$ are lower respectively upper bounds for $g^{(\ell)}(\eps)$.
\end{proof}

Now, fix $m\ge3$, and let $\eps_0=\eps_0(m)>0$ be given as in Lemma~\ref{lma:hDetailed}. We first show the 
second inequality. Since $L$ approaches $\frac{1}{2}$ continuously, we
can choose $a_0>0$ so that $L(a_0) =\frac{1}{2}-\frac{\eps_0}{9}$. Given $x\ge 1$, 
let $n_x:=\lfloor \log_{\gamma}\frac{x}{a_0}\rfloor$. We first
restrict to $x$'s which are sufficiently large so that $n_x\ge 1$. This immediately yields 
$$
a_0\le x\gamma^{-n_x}.
$$
Using (\ref{eq:hLIterate}) and monotonicity of $L$ and $h$, we have
$$
L(x)\,=\,h^{(n_x)}\big(L(x\gamma^{-n_x})\big)\,\ge\,
h^{(n_x)}\big(L(a_0)\big).
$$
By Lemma~\ref{lma:hDetailed} and the definition of $a_0$, 
we have that
$$
\Pr(W_m\ge x)\,=\,\tfrac{1}{2}-L(x)\,\le\,\tfrac{1}{2}-h^{(n_x)}\big(L(a_0)\big)\,=\, \tfrac{1}{2}-h^{(n_x)}\big(\tfrac{1}{2}-\tfrac{\eps_0}{9}\big)\,\le\, \eps_0^{(\frac{m+1}{2})^{n_x}}.
$$
An easy computation shows that, for all $x$ for which
$n_x\ge 1$ %and with $\beta(m)$ as defined in~\eqref{eq:beta},
$$
\left(\frac{m+1}{2}\right)^{n_x}\ge\,\frac{2}{m+1}\left(\frac{x}{a_0}\right)^{\beta(m)}.
$$
From this, the upper bound follows with $c_2=-\frac{2}{m+1}a_0^{-\beta(m)}\log\eps_0$ for all large $x$. By decreasing $c_2$ if necessary, one can of course get the desired inequality for all $x\ge1$.

We now move to the lower bound. This time, choose $a_0>0$ so that $L(a_0) =\frac{1}{2}-\frac{\eps_0}{2}$, and given $x\ge 1$, 
let $n_x:=\lceil \log_{\gamma}\frac{x}{a_0}\rceil$. We again
restrict to large $x$ for which $n_x\ge 1$. This immediately yields 
$$
a_0\ge x\gamma^{-n_x}.
$$
Using (\ref{eq:hLIterate}) and monotonicity of $L$ and $h$, we have
$$
L(x)\,=\,h^{(n_x)}\big(L(x\gamma^{-n_x})\big)\,\le\, h^{(n_x)}\big(L(a_0)\big).
$$
By Lemma \ref{lma:hDetailed}, we have that
$$
\Pr(W_m\ge x)\,=\, \tfrac{1}{2}-L(x)\,\ge\, \tfrac{1}{2}-h^{(n_x)}\big(L(a_0)\big)\,=\,
\tfrac{1}{2}-h^{(n_x)}\big(\tfrac{1}{2}-\tfrac{\eps_0}{2}\big)\,\ge\, \big(\tfrac{\eps_0}{2}\big)^{(\frac{m+1}{2})^{n_x}}.
$$
An easy computation shows that one has, for all $x$ for which
$n_x\ge 1$, that
$$
\left(\frac{m+1}{2}\right)^{n_x}\le\, \frac{m+1}{2} \left(\frac{x}{a_0}\right)^{\beta(m)}.
$$
From this, the lower bound follows for some $c_1$ for all large $x$, and 
by increasing $c_1$ if necessary, one can of course get the desired inequality for all $x\ge1$. This proves part~\emph{b)}.

\subsection{Proof of part~\emph{c)}}

We set out to show that $\beta(m)$ is strictly increasing. Since
\be\label{eq:betadiff}
\beta(m+2)-\beta(m)\,=\,\frac{\log\frac{m+3}{2}\log\gamma(m)-\log\frac{m+1}{2}\log\gamma(m+2)}{\log\gamma(m)\log\gamma(m+2)},
\ee
it will suffice to show that the numerator in the right-hand side of~\eqref{eq:betadiff} is strictly positive. Using the recursive structure of $\gamma(m)$, i.e., that $\gamma(m+2)=\frac{m+2}{m+1}\,\gamma(m)$, we aim to show that
$$
\log\gamma(m)\log\frac{m+3}{m+1}-\log\frac{m+1}{2}\log\frac{m+2}{m+1}\,>\,0.
$$
For $x\ge0$ a Taylor estimate for $\log(1+x)$ gives the lower and upper bounds $x-\frac{x^2}{2}$ and $x$, respectively. A lower bound on the numerator in the right-hand side of~\eqref{eq:betadiff} is thus given by
\be\label{eq:betalower}
\log\gamma(m)\bigg[\frac{2}{m+1}-\frac{2}{(m+1)^2}\bigg]-\frac{1}{m+1}\log\frac{m+1}{2}.
\ee
A lower bound on $\gamma(m)$ can be obtained from known bounds on the central binomial coefficient. For instance, Wallis' product formula states that $a_n:=\prod_{k=1}^n\frac{2k}{2k-1}\frac{2k}{2k+1}$ converges to $\frac{\pi}{2}$ as $n\to\infty$.
%(see e.g.~\cite{wastlund07} for an elementary proof).
Since $a_n$ is increasing we have $a_n\le\frac{\pi}{2}$ for all $n\ge1$,
leading to the bound
$$
\binom{2n}{n}\,\ge\, 4^n\sqrt{\frac{2}{\pi(2n+1)}}.
$$
Consequently $\gamma(m)\ge\sqrt{2m/\pi}$ for all $m\ge3$. After multiplication by $m+1$, a lower bound on the expression in~\eqref{eq:betalower} is given by
$$
\log\frac{2m}{\pi}\bigg[1-\frac{1}{m+1}\bigg]-\log\frac{m+1}{2}\,=\,\log\frac{4}{\pi}-\log\frac{m+1}{m}-\frac{1}{m+1}\log\frac{2m}{\pi}.
$$
One may check that the latter expression is increasing in $m$ and positive for $m=13$. Using the slightly sharper lower bound $\frac{4^n}{\sqrt{\pi n}}(1-\frac{1}{8n})$ on the central binomial coefficient, obtained from the Stirling series, one may arrive at an alternative lower bound on the difference in~\eqref{eq:betadiff}, which is positive for all $m\ge5$. In either case, one further checks that $\beta(m)<\beta(m+2)$ for the remaining values of $m$ by hand, so that $\beta(m)$ is strictly increasing for all $m\ge3$.

\subsection{Proof of part~\emph{d)}}
Since $m$ will now be changing, it is natural to now write $L_m$ instead of $L$. 
The distribution given by $F_m(x)=\frac{1}{2}+L_m(x)$ has density $L_m'(x)$. 
We will show that $\lim_{m\to\infty}L_m'(x)=e^{-\pi x^2}$, which we recognize as the 
density of a centered normal distribution with variance $1/(2\pi)$. 
By virtue of Scheff\'e's theorem (see e.g.\ \cite{durrett10}), pointwise 
convergence of densities implies the desired weak convergence of $F_m$ to a normal distribution.
By symmetry, it suffices to prove this for $x\ge 0$ which we now assume to be the case.

We first show that $L_m'(x\gamma^{-1})\to1$ as $m\to\infty$, for every $x\ge 0$. First recall that 
by Proposition~\ref{prop:L}, we have that
$L_m'(x)\le1$ for all $x\ge 0$. Using~\eqref{eq:Lprime} and~\eqref{eq:gprime} and then
that $L_m'\le1$, we find that
$$
L_m'(x\gamma^{-1})\,=\,\prod_{k=1}^\infty\left[1-4[L_m(x\gamma^{-k-1})]^2\right]^{\frac{m-1}{2}}\,\ge\,\prod_{k=1}^\infty\left[1-4x^2\gamma^{-2(k+1)}\right]^{\frac{m-1}{2}}.
$$
Since $e^{-2y}\le 1-y$ for small positive $y$, 
replacing the terms of the product by an exponential, we obtain that for all large $m$
\bea
\begin{aligned}
L_m'(x\gamma^{-1})\;&\ge\;\prod_{k=1}^\infty \exp\left(-4x^2(m-1)\gamma^{-2(k+1)}\right)\;=\;\exp\left(-4x^2(m-1)\sum_{k=1}^\infty\gamma^{-2(k+1)}\right)\\
&\ge\;\exp\left(-8x^2m\gamma^{-4}\right),
\end{aligned}
\eea
where we in the last step have used that $\gamma^2\ge2$. Since $\gamma=\gamma(m)$ increases at the rate of $\sqrt{m}$, we may conclude that $L_m'(x\gamma^{-1})\to1$ as $m\to\infty$. Moreover, the convergence is uniform in $x$ over compact sets. Consequently, for every $\eps>0$ and $x\ge 0$ 
there is $m_0$ such that $L_m(x\gamma^{-1})\ge(1-\eps)x\gamma^{-1}$ for all $m\ge m_0$.

Second, again using~\eqref{eq:Lprime} and~\eqref{eq:gprime}, or differentiating~\eqref{eq:hLIterate}, we arrive at
$$
L_m'(x)\,=\,L_m'(x\gamma^{-1})\left[1-4[L_m(x\gamma^{-1})]^2\right]^{\frac{m-1}{2}}.
$$
Together with our previous conclusions we find that
$$
L_m'(x\gamma^{-1})\left[1-4x^2\gamma^{-2}\right]^\frac{m-1}{2}\,\le\,L_m'(x)\,\le\,L_m'(x\gamma^{-1})\left[1-(1-\eps)^24x^2\gamma^{-2}\right]^\frac{m-1}{2}.
$$
Taking limits, first as $m\to\infty$ and then as $\eps\to0$, leaves us with
$$
\lim_{m\to\infty}L_m'(x)=e^{-\pi x^2},
$$
as required.

\section{All measures are distributional limits}\label{s.DistLimits}

In this section, we prove Theorem~\ref{thm:newLIMITS} and Propositions~\ref{prop:newNoLIMITS} and~\ref{prop:NoSubseq}.

\begin{proof}[Proof of Theorem~\ref{thm:newLIMITS}]
The main part of the proof will be to, in both settings~\emph{a)} and~\emph{b)}, prove the result for a restricted class of probability measures, namely those $\mu$ of the form $\sum_{i=1}^k q_i \delta_{x_i}$, i.e., having finite support.
To see why this will suffice in order to obtain the general result, we first state a simple lemma, whose proof is left to the reader, 
concerning metric spaces. Assume, in a metric space, we are given $x_m$ converging to $x_\infty$ and for 
each $m$, we have $x_{m,n}$ converging to $x_m$ as $n\to\infty$. Then there is a sequence $m_n$ (not 
necessarily strictly increasing) so that we have that $x_{m_n,n}$ converges to $x_\infty$ as $n\to\infty$. 

We note that it is well known that convergence in distribution is metrizable. 
Assume now that we are given an arbitrary probability measure $\mu$ and a sequence $(a_n)_{n\ge1}$
satisfying the stated properties. It is clear we can find a sequence $(\mu_m)_{m\ge1}$, each with finite support as above,
converging to $\mu$. Assume that, for each $m$, we can find a sequence of
Boolean functions $(f_{m,n})_{n\ge1}$ such that $f_{m,n}$ is defined on $n$ bits and 
$a_n(T(f_{m,n})-\frac{1}{2})$ approaches, as $n\to\infty$, $\mu_m$ in distribution. 
By the above general metric space result, there exists a sequence $(m_n)_{n\ge1}$ (not necessarily strictly increasing)
so that $a_{n}(T(f_{m_n,n})-\frac{1}{2})$ approaches $\mu$ in distribution, as $n\to\infty$. This shows that it will suffice to prove parts~\emph{a)} and~\emph{b)} for measures $\mu$ having finite support.

\emph{Proof of part~a)}.
Assume that a probability measure $\mu$ of the form $\sum_{i=1}^k q_i \delta_{x_i}$ is given and let 
$(a_n)_{n\ge1}$ satisfy $1\ll a_n\ll\sqrt{n}$. We may assume that $x_1<x_2<\ldots<x_k$ and that the 
$q_i$'s are all positive.
For each $i=1,2,\ldots,k$ fix $y_i\in\R$ such that
$$
1-\Phi(y_i)\,=\,q_1+q_2+\ldots+q_i,
$$ 
where $\Phi(\cdot)$ denotes the distribution function of the standard Gaussian. (Of course, this 
defines $y_k$ to be $-\infty$, but we allow this slight abuse of notation.)

Now let $E_i$ denote the event that the proportion of 1's among the $n$ bits is at least 
$\frac{1}{2}+x_i/a_n$, and let $F_i$ denote the event that the proportion of 1's among the first $\lfloor a_n\rfloor$ bits is at least $\frac{1}{2}+y_i/(2\sqrt{a_n})$.
Although not explicit in the notation, these events depend on $n$.
Notice further that the events are defined so that $E_1\supseteq E_2\supseteq\ldots\supseteq E_k$ and $F_1\subseteq F_2\subseteq\ldots\subseteq F_k$ hold.
Finally, we define $f_n$ as the indicator function of the event $\bigcup_{i=1}^k(E_i\cap F_i)$.

Let $p_n=\frac{1}{2}+x/a_n$. To complete the proof of part~\emph{a)} we need to verify that
$$
\Pr\big(a_n(T_n-\tfrac{1}{2})\le x\big)\,=\,\Pr_{p_n}\big(E_i\cap F_i\text{ for some }i\big)
$$
tends to $0$, $\sum_{i=1}^jq_i$ or $1$, depending on whether $x<x_1$, $x\in(x_j,x_{j+1})$  and $j=1,2,\ldots,k-1$, or $x>x_k$. We first examine the events $E_i$ and $F_i$. Appealing to the Lindeberg-Feller central limit theorem, or Chebyshev's inequality, we find that
\be\label{eq:Eprobab3}
\Pr_{p_n}(E_i)\,=\,\Pr_{p_n}\bigg(\frac{1}{n}\sum_{j\in[n]}\omega_j\ge\frac{1}{2}+\frac{x_i}{a_n}\bigg)\,\to\,\left\{
\begin{aligned}
0 && x<x_i,\\
1 && x>x_i,
\end{aligned}
\right.
\ee
as by assumption $a_n\ll\sqrt{n}$. Moreover, as $a_n\gg1$, using Lindeberg-Feller, for any $x\in\R$
\be\label{eq:Fprobab3}
\Pr_{p_n}(F_i)\,=\,\Pr_{p_n}\bigg(\frac{1}{\lfloor a_n\rfloor}\sum_{j=1}^{\lfloor a_n\rfloor}\omega_j\ge\frac{1}{2}+\frac{y_i}{2\sqrt{a_n}}\bigg)\,\to\,1-\Phi(y_i).
\ee
(The above abuse of notation is here manifested in that $F_k$ equals the whole sample space.)

Since $\Pr_{p_n}\big(E_i\cap F_i\text{ for some }i\big)$ is at most $\Pr_{p_n}\big(E_i\text{ for some }i\big)$, the case $x<x_1$ is immediate from~\eqref{eq:Eprobab3}. Similarly, as $F_k$ equals the whole sample space, $\Pr_{p_n}(E_k)$ gives a lower bound on $\Pr_{p_n}\big(E_i\cap F_i\text{ for some }i\big)$, so also the case $x>x_k$ is immediate from~\eqref{eq:Eprobab3}. For $j=1,2,\ldots,k-1$ and $x\in(x_j,x_{j+1})$, using~\eqref{eq:Eprobab3},~\eqref{eq:Fprobab3} and the fact that $F_1\subseteq F_2\subseteq\ldots\subseteq F_k$, gives
\be\label{eq:planedlimit}
\lim_{n\to\infty}\Pr_{p_n}\big(E_i\cap F_i\text{ for some }i\big)\,=\,
\lim_{n\to\infty}\Pr_{p_n}(F_j)\,=\,q_1+q_2+\ldots+q_j,
\ee
due to the definition of the $y_i$'s, as required.

\emph{Proof of part~b)}.
Assume again that a probability measure $\mu$ of the form $\sum_{i=1}^kq_i\delta_{x_i}$ is given and let 
$(a_n)_{n\ge1}$ satisfy $\log n\ll a_n\ll\sqrt{n}$. We may assume that $x_1<x_2<\ldots<x_k$ and that the 
$q_i$'s are all positive. As before, let $E_i$ denote the event that the proportion of 1's among the $n$ bits is at least $\frac{1}{2}+x_i/a_n$.

Let $\ell_n=\lfloor2\log_2n\rfloor$. In order to define events $F_i$ we extend the partial ordering on binary strings of length $\ell_n$, in which $y\ge y'$ if $y$ dominates $y'$ coordinate-wise, to a total ordering. (One such ordering is inherited from the set of integers through their binary representation.)
%The partial ordering of $\{0,1\}^n$, in which $\omega\ge\omega'$ if $\omega$ dominates $\omega'$ coordinate-wise, is naturally extended to a total ordering inherited from the set of integers through their binary representation. We will assume this ordering in our construction of events $F_i$. To do this, let $\ell_n=\lfloor2\log_2n\rfloor$ and $y\in\{0,1\}^{\ell_n}$ be given.
Consider the $n$ bits of $\omega$ positioned in a circle and, for each string $y\in\{0,1\}^{\ell_n}$, let $F(y)$ denote the event that $\omega$ contains an interval of length $\ell_n$ on which $\omega$ is at least as large as $y$ with respect to the total ordering.
We claim that it is possible to choose $y_1^n\ge y_2^n\ge\ldots\ge y_k^n$ in $\{0,1\}^{\ell_n}$ such that $F_i:=F(y_i^n)$, for every $i=1,2,\ldots,k$, $x\in\R$ and with $p_n=\frac{1}{2}+x/a_n$, satisfies
\be\label{eq:y-claim}
\lim_{n\to\infty}\Pr_{p_n}(F_i)\,=\,q_1+q_2+\ldots+q_i.
\ee
These definitions guarantee that $E_1\supseteq E_2\supseteq\ldots\supseteq E_k$ and $F_1\subseteq F_2\subseteq\ldots\subseteq F_k$ hold also here.

Assuming that such a selection of $y_i^n$'s can be made, we define $f_n$ as the indicator function of the event $\bigcup_{i=1}^k(E_i\cap F_i)$. The resulting function is monotone (with respect to the partial ordering on $\{0,1\}^n$) and invariant with respect to rotations of the bits and therefore transitive.

Before proving the claim we argue for why this would complete the proof of part~\emph{b)}. As in part~\emph{a)}, to complete the proof we need to verify that
$$
\Pr\big(a_n(T_n-\tfrac12)\le x\big)\,=\,\Pr_{p_n}\big(E_i\cap F_i\text{ for some }i\big)
$$
tends to $0$, $\sum_{i=1}^jq_i$ or $1$, depending on whether $x<x_1$, $x\in(x_j,x_{j+1})$ and $j=1,2,\ldots,k-1$, or $x>x_k$. As the events $E_i$ are defined just as above,
and the monotone structure $F_1\subseteq F_2\subseteq\ldots\subseteq F_k$ holds also here, these conclusions would follow from the claim as in part~\emph{a)}.

It remains to prove the claim. We first show that for every $n\ge1$ and $q\in[0,1]$ there exists $y\in\{0,1\}^{\ell_n}$ such that $\big|\Pr_{1/2}\big(F(y)\big)-q\big|\le2/n$. Let $y_1=\max\{y:\Pr_{1/2}\big(F(y)\big)\ge q\}$ and denote its successor in the total ordering by $y_2$, should it exist (otherwise set $F(y_2)=\emptyset$). Then $\Pr_{1/2}\big(F(y_2)\big)<q$ and we have
$$
0\,\le\,\Pr_{1/2}\big(F(y_1)\big)-\Pr_{1/2}\big(F(y_2)\big)\,\le\,\Pr_{1/2}\big(\text{some interval of $\omega$ equals }y_1\big)\,\le\, n2^{-\ell_n}\,\le\, 2/n,
$$
where the second-to-last inequality comes from a first moment estimate. In particular, for $y=y_1$ we obtain $\big|\Pr_{1/2}\big(F(y)\big)-q\big|\le2/n$.

We next show that if some (sequence of) $y\in\{0,1\}^{\ell_n}$ satisfies $\Pr_{1/2}\big(F(y)\big)\to q$ as $n\to\infty$, then also $\Pr_{p_n}\big(F(y)\big)\to q$. Recall the monotone coupling $(\eta_p)_{p\in[0,1]}$ of elements in $\{0,1\}^n$. Assume that $x\ge0$ so that $p_n\ge1/2$, and let $A$ denote the event that $\eta_{1/2}$ contains no interval on which it is at least as large as $y$ (in the total ordering) but that $\eta_{p_n}$ does.
%Let further $\kappa$ denote the least $k\in[n]$ such that $\eta_{p_n}$ is at least as large as $y$ on the interval starting at $k$, and let $B$ denote the event that $\xi_i\in(\tfrac{1}{2},p_n]$ for some $i\in\{\kappa,\kappa+1,\ldots,\kappa+\ell_n-1\}$.
Then,
\be\label{eq:F(y)-bound}
\Pr_{p_n}\big(F(y)\big)-\Pr_{1/2}\big(F(y)\big)\,=\,\Pr(A)\,=\,\E\big[\Pr(A|\eta_{p_n})\big]\,\le\,\ell_n(p_n-\tfrac12)/p_n\,\le\,2x\ell_n/a_n,
\ee
since $\Pr(A|\eta_{p_n})$ either equals zero, in the case $\eta_{p_n}$ does not contain an interval on which it is at least as large as $y$, or is bounded by $\ell_n(p_n-\frac12)/p_n$, since if such an interval exists then at least one bit in this interval (or the first in some ordering, if there are several intervals with this property) must have changed its value as $p$ ranges from $\tfrac12$ to $p_n$. Since $a_n\gg\log n$, the upper bound in~\eqref{eq:F(y)-bound} tends to zero as $n\to\infty$, as claimed. The case $x<0$ is analogous.

This proves the claim that we may choose $y_i^n$'s such that~\eqref{eq:y-claim} holds for all $x\in\R$, and thus completes the proof of part~\emph{b)}.
\end{proof}

\begin{remark}
In order to obtain a sequence $(f_n)_{n\ge1}$ of monotone graph properties on $n$ vertices and ${n\choose 2}$ edges (which are the variables here), with the property described in the remark following Theorem~\ref{thm:newLIMITS}, only minor modifications to the above construction are necessary: Define the $E_i$'s as before, but now on ${n\choose 2}$ bits. Let $\ell_n=\lfloor4\log_2n\rfloor$, say, and extend the usual partial ordering on the set of unlabled graphs on $\ell_n$ vertices to a total ordering. Given an unlabled graph $y$ on $\ell_n$ vertices, define $F(y)$ as the event that $\omega$ contains an induced subgraph at least as large as $y$ with respect to the total ordering. The calculations needed to verify that one may choose graphs $y_i^n$ so that~\eqref{eq:y-claim} holds also in this setting are straightforward.
\end{remark}

We continue investigating the possible behavior of the scaling coefficients.

\begin{proof}[Proof of  Proposition~\ref{prop:newNoLIMITS}]
Let $\mu$ be a non-degenerate probability measure and $(f_n)_{n\ge1}$ some sequence of monotone Boolean functions for which $a_n(T_n-b_n)$ approaches $\mu$ in distribution. Denote by $F$ the distribution function associated to $\mu$.

To prove part~\emph{a)}, we first recall the well known fact that for monotone Boolean functions the total influence, for $p$ bounded away from 0 and 1, is of order at most $\sqrt{n}$. The Margulis-Russo formula then implies that there is an upper bound of order $\sqrt{n}$ on the derivative of $\Pr_p(f_n=1)$, for $p$ bounded away from 0 and 1. Now, pick $x_1<x_2$ at which $F$ is continuous and satisfies $F(x_1)\le1/3$ and $F(x_2)\ge2/3$. With $p_i=b_n+x_i/a_n$, we thus obtain, for all large $n$,
\be\label{eq:average}
\frac{a_n}{4(x_2-x_1)}\,\le\,\frac{\Pr_{p_2}(f_n=1)-\Pr_{p_1}(f_n=1)}{p_2-p_1}\,\le\,\frac{a_n}{x_2-x_1}.
\ee
Via the mean value theorem, the first inequality in~\eqref{eq:average} gives a lower bound on the derivative of $\Pr_p(f_n=1)$ for some $p\in[p_1,p_2]$, which, consequently, yields an upper bound on $a_n$ of order $\sqrt{n}$.

The proof of part~\emph{b)} is similar to that of part~\emph{a)}. Recall the well known fact from KKL~\cite{kahkallin88} and its extensions that the total influence of transitive Boolean functions grows at least of order $\log n$ when their variance is bounded away from 0 and 1. With the mean value theorem, the Margulis-Russo formula and the second inequality in~\eqref{eq:average} we thus obtain a corresponding lower bound on $a_n$, of order $\log n$.

The proofs for parts~\emph{c)},~\emph{d)} and~\emph{e)} are also similar to each other. Let, again, $x_1<x_2$ be continuity points of $F$ and set $p_i=b_n+x_i/a_n$. Using the mean value theorem we find $q_n\in[p_1,p_2]$ so that
\be\label{eq:Flimit}
\frac{F(x_2)-F(x_1)}{x_2-x_1}\,=\,\lim_{n\to\infty}\frac{\Pr_{p_2}(f_n=1)-\Pr_{p_1}(f_n=1)}{x_2-x_1}\,=\,\lim_{n\to\infty}\frac{1}{a_n}\,\frac{d}{dp}\Pr_p(f_n=1)\big|_{q_n}.
\ee
For~\emph{c)} we assume that $(a_n)_{n\ge1}$ is bounded above and that $0<F(x_1)\le F(x_2)<1$. Recall the discrete Poincar\'e inequality which states that
for any Boolean function $f:\{0,1\}^n\to\{0,1\}$
\be\label{eq:poincare}
\sum_{i=1}^n\Inf_i^p(f)\,\ge \,\Var_p(f).
\ee
Hence, a lower bound on~\eqref{eq:Flimit} of $\liminf_{n\to\infty}\Var_{q_n}(f_n)/(\sup_{n\ge1}a_n)$, which is strictly positive, is obtained via the Margulis-Russo formula. Consequently $F(x_1)<F(x_2)$, from which~\emph{c)} follows.

For transitive functions KKL~\cite{kahkallin88} and its extensions improve on the lower bound in~\eqref{eq:poincare} with a factor of order $\log n$. In this setting, assuming that $(a_n/\log n)_{n\ge1}$ is bounded above, we thus obtain a lower bound on~\eqref{eq:Flimit} of $\liminf_{n\to\infty}\Var_{q_n}(f_n)\big(\inf_{n\ge1}\log n/a_n\big)$, which is strictly positive by assumption. Again we find that $F(x_1)<F(x_2)$, which settles part~\emph{d)}. 

For~\emph{e)} we assume that $a_n/\sqrt{n}$ is bounded away from 0 and use the upper bound of order $\sqrt{n}$ on the derivative of $\Pr_p(f_n=1)$, for $p$ bounded away from 0 and 1, to obtain a uniform upper bound in~\eqref{eq:Flimit}, showing that $F$ is Lipschitz continuous. 
\end{proof}

\begin{remark}
While the above proof of part~\emph{c)} shows that no monotone Boolean function $f$ may have $\Pr_p(f=1)$ making a finite
number of ``jumps'' but otherwise remaining more or less constant, it is still possible to have
sudden jumps between which $\Pr_p(f=1)$ grows linearly. An example would be the event $A$ consisting 
of all configurations for which either $\omega_1=1$ and the proportion of 1's is at least $1/3$ 
or the proportion of 1's is at least $2/3$.
\end{remark}

As mentioned in the introduction, one easily shows that no subsequence of the probability measures giving equal weight to the points in $\Omega_m=\{\pm2^k:k=1,2,\ldots,m\}$ can be normalized in order to give a nondegenerate limit. We base the proof of Proposition~\ref{prop:NoSubseq} on this example.

\begin{proof}[Proof of Proposition~\ref{prop:NoSubseq}]
We prefer to work with continuous distributions. Let $\mu_m$ be the measure whose density function equals $\frac{1}{2m}$ for $x\in[k-\frac{1}{2},k+\frac{1}{2}]$ and $k\in\Omega_m$, and 0 otherwise; let $F_m$ denote the corresponding distribution function. Then $F_m$ is continuous and $\mu_m$ effectively has the same properties as the uniform measure on $\Omega_m$. According to Theorem~\ref{thm:newLIMITS} we may choose $a_n=n^{1/4}$, say, and monotone Boolean functions $f_{m,n}:\{0,1\}^n\to\{0,1\}$ such that $a_n(T(f_{m,n})-\frac{1}{2})$ tends to $\mu_m$ in distribution, as $n\to\infty$. Writing $F_{m,n}$ for the distribution function of $a_n(T(f_{m,n})-\frac{1}{2})$ and using that $F_m$ is continuous, we find for each $m$ an integer $n_m$ such that 
\be\label{eq:Funiform}
\sup_{x\in\R}\big|F_{m,n}(x)-F_m(x)\big|\,\le\,\frac{1}{m}\quad\text{for all }n\ge n_m.
\ee
Define $f_n:=f_{m,n}$ for $n\in[n_m,n_{m+1})$, and note that $f_n$ is a monotone function on $n$ variables.

Let $m_n:=\max\{m\in\N:n_m\le n\}$. Now, assume there are nonnegative sequences $(b_n)_{n\ge1}$ and $(c_n)_{n\ge1}$, and a nondegenerate probability measure with distribution function $F$ such that, along some subsequence, $F_{m_n,n}(c_nx+b_n)\to F(x)$ for all continuity points of $F$. Then,
$$
\big|F_{m_n}(c_nx+b_n)-F(x)\big|\,\le\,\big|F_{m_n}(c_nx+b_n)-F_{m_n,n}(c_nx+b_n)\big|+\big|F_{m_n,n}(c_nx+b_n)-F(x)\big|,
$$
which, for continuity points of $F$, would tend to zero along this subsequence, in virtue of~\eqref{eq:Funiform}. This would contradict the fact that no subsequence of $\mu_m$ can be normalized to obtain a nondegenerate limit, and therefore shows that no subsequence of $(T(f_n))_{n\ge1}$ can be normalized to obtain a nondegenerate limit.
\end{proof}

\appendix
\numberwithin{equation}{section}
\numberwithin{figure}{section}
\numberwithin{thm}{section}

\def\Spec{\mathsf{Spec}}
\def\lora{\longrightarrow}

\section{The tail of the crossing probability in near-critical percolation\\ (An appendix by G\'abor Pete)}

The goals of this appendix are to prove Theorem~\ref{thm:Perc} of the main text, to draw attention to an interesting difference between the flip times of crossing events in near-critical and dynamical percolation in the plane, and to examine, on an intuitive level, how near-critical high-dimensional percolation may behave.

\subsection{Planar percolation}

We will work with site percolation on the triangular lattice $\TG$ with mesh size 1, at density close to the critical value $p_c=1/2$. See~\cite{grimmett99,werner09} for background.  Let $\LR_\Quad$ denote the left-to-right crossing event in a nice quad $\Quad$, by which we mean the image of the square $[0,1]^2$ under a smooth injective map into $\C$. If we magnify $\Quad$ by a factor of $\rho$, the new quad will be denoted by $\rho\Quad$, the center of magnification being irrelevant. Furthermore, let $\alpha_4(n)$ denote the critical alternating four-arm probability from a given site to Euclidean distance $n$, and let $r(n):=1/\big(n^2\alpha_4(n)\big)$. The asymptotics $r(n)=n^{-3/4+o(1)}$ was proved in \cite{smiwer01}.

Consider now the canonical coupling of percolation configurations at different densities: take $U_x\sim\mathsf{Unif}[0,1]$ i.i.d.~for all vertices $x\in\TG$, and for any $\lambda\in(-\infty,\infty)$ and $n$ large enough so that $\lambda r(n)\in [-1/2,1/2]$, define $\omega_n(\lambda)$ to be the configuration where $x$ is open if{f} $U_x \leq 1/2 + \lambda r(n)$. This process  $\{\omega_n(\lambda)\}_{\lambda\in\R}$ is called the near-critical ensemble. Note that the flip time $T_n$ for $\LR_{[0,n]^2}$ considered in the main text, with rescaling $W_n:=n^2\alpha_4(n)(T_n-1/2)$, exactly satisfies $\{W_n \leq \lambda\} = \{\omega_n(\lambda) \in \LR_{[0,n]^2}\}$.

\begin{figure}[htbp]
\begin{center}
\includegraphics[width=0.4\textwidth]{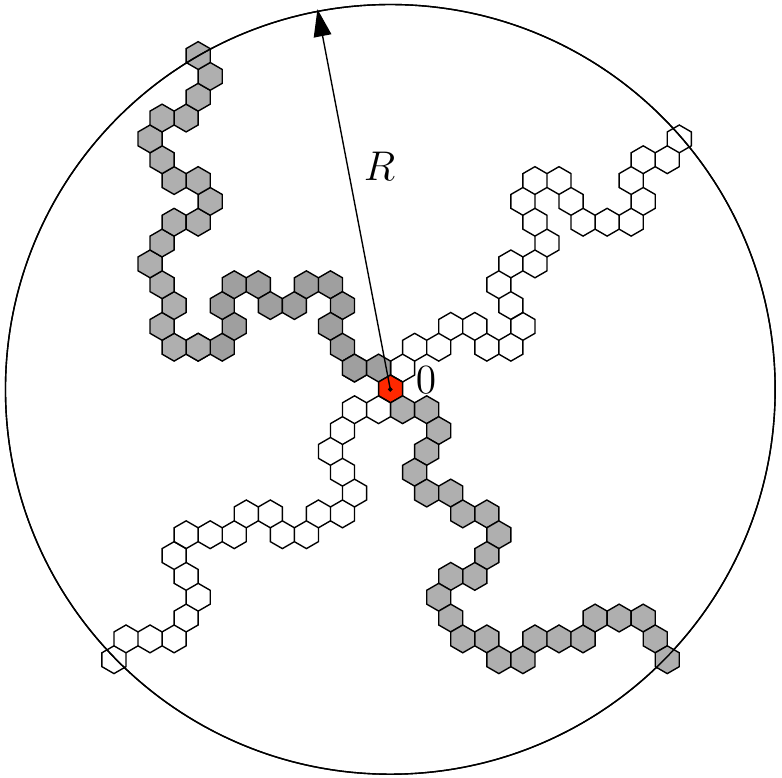}
\end{center}
\caption{A realization of the {\it four-arm} event.}
\label{f.armsevents}
\end{figure}

It is proved in~\cite{garpetsch13,garpetsch} that the process $\{\omega_n(\lambda)\}_{\lambda\in\R}$ has a scaling limit in a certain topology: for any fixed $\lambda$, the static topology encodes quad-crossings, while as a process, the Skorokhod topology of c\`adl\`ag processes is used; see \cite[Theorem 1.5]{garpetsch} for the precise statement. (The time-parametrization of $\omega_n(\lambda)$ in \cite{garpetsch} is slightly different from the above definition, but this makes no real difference.) It is also proved in \cite[Proposition 9.6]{garpetsch} that this limiting c\`adl\`ag process has no jumps at deterministic values of $\lambda$, which implies that for any $\lambda\in\R$, the limit
\begin{equation}\label{e.flamQ}
f(\lambda,\Quad):=\lim_{n\to\infty} \Pbo{p_c+\lambda r(n)}{\LR_{n\Quad}}
\end{equation}
exists for any nice quad $\Quad$ (which does not automatically follow from the existence of the limit as a process, because of the Skorokhod topology of c\`adl\`ag processes). Moreover, this limit is absolutely continuous in $\lambda$, and most importantly, it is non-trivial: it satisfies $f(\lambda,\Quad)\in (0,1)$, and 
\begin{equation}\label{e.lamlim}
\lim_{\lambda\to -\infty} f(\lambda,\Quad) = 0\,,\quad\textrm{and}\quad \lim_{\lambda\to \infty} f(\lambda,\Quad) = 1\,.
\end{equation}
In fact, these properties were already known from Kesten's work \cite{kesten87,nolin08}, for any subsequential limit, at that time. Briefly, in the entire critical window where $\Pso{p}{{\LR_{n\Quad}}} \in (\eps,1-\eps)$ holds, the expected number of pivotals is comparable to the scaling factor $n^2\alpha_4(n)$ (with constant factors depending on $\Quad$ and $\eps$), which implies absolute continuity and~(\ref{e.lamlim}) using Russo's formula.

The above results imply Part~(a) of Theorem~\ref{thm:Perc}.

For the proof of Part~(b), we will also need that the limit $f(\lambda,\Quad)$ in (\ref{e.flamQ}) is conformally covariant, proved in \cite[Theorem 10.3]{garpetsch}. Instead of defining exactly what this means, let us just give a special case that we will use:
\begin{equation}\label{e.scalecov}
f(\rho\lambda,\Quad)=f(\lambda,\rho^{4/3}\Quad)\,,
\end{equation}
for any scaling factor $\rho>0$. For simplicity, we will just take $\Quad=[0,1]^2$. Then, the limiting tail behaviour of the rescaled flip time $W_n$ is determined by the following result, proved below, together with the well-known duality $f(-\lambda,[0,1]^2)=1-f(\lambda,[0,1]^2)$:
 
\begin{thm}\label{t.tail}
As $\lambda\to\infty$, we have the superexponential decay 
$$f(-\lambda,[0,1]^2) = \exp\left(-\Theta\big(\lambda^{4/3}\big)\right),$$
where, as usually, $g(\lambda)=\Theta(h(\lambda))$ means the existence of universal constants $0<c<C<\infty$ such that $c < g(\lambda)/h(\lambda) < C$ holds for all $\lambda$ in question.
\end{thm}

Besides the question raised in the main text, another motivation for Theorem~\ref{t.tail} is \cite{hammospet12}, where the analogous tail behaviour was studied for the scaling limit of dynamical percolation. Namely, if we start with critical percolation, then resample each site at rate $r(n)$, keeping the configuration stationary, then we may look at 
\begin{equation}\label{e.gtQ}
g(t,\Quad):=\lim_{n\to\infty} \Pb{\LR_{n\Quad}\text{ does not hold at any moment in }[0,t]}\,.
\end{equation}
Again, this limit exists and is conformally covariant by \cite{garpetsch13,garpetsch}. Then, regarding the tail behaviour, it was proved in \cite{hammospet12} using general Markov chain arguments such as spectral computations and a dynamical (space-time) FKG-inequality, that there exists an absolute constant $c>0$, and for every $K>0$ some $c_K>0$, such that 
\begin{equation}\label{e.dynexit}
\exp(-c\, t) \leq g(t,[0,1]^2) \leq c_K t^{-K}\,,
\end{equation}
for all $t\geq 1$. Furthermore, the present author was speculating in \cite{pete-talk}, using non-rigorous renormalization ideas (motivated by \cite{lanlaf94,langlands05,schsmi11}) and a very strong universality hypothesis, that the true behaviour could be $\exp(-t^{2/3+o(1)})$. Several people in the community agreed that this speculation looked quite solid (even if non-rigorous) as a lower bound, while more questionable as an upper bound. And, as typical for these planar percolation scaling limits, that argument seemed to be working equally well for the symmetric (dynamical) and asymmetric (near-critical) versions. However, our present Theorem~\ref{t.tail} violates not only this bold prediction for the near-critical case, but it also shows that $f(t,\Quad)$ does not satisfy the rigorous exponential lower bound of~(\ref{e.dynexit}) for $g(t,\Quad)$, hence this tail probability question turns out to be an instance where the asymmetric versus symmetric dynamical versions of critical percolation show drastically different behaviour. Regarding the true decay in the symmetric dynamical version, our simulations suggest a subexponential decay, but are far from being conclusive, and are even further from giving a prediction for the exponent. See Figure~\ref{f.simu}.

%%%%%%%%%%%%%%%%%%%%%%%%%%
\begin{figure}[htbp]
%\SetLabels
%(0.08*0.4)$\tau$\\
%(0*0.7)$\xi_1$\\
%(0*0.3)$\xi_2$\\
%(0.16*0.3)$\xi_3$\\
%(0.16*0.7)$\xi_4$\\
%(0.83*0.2)$x_1$\\
%(0.72*0.1)$x_2$\\
%(0.96*0.17)$x_3$\\
%(0.93*0.8)$x_4$\\
%\endSetLabels
%\ShowGrid
\centerline{
\AffixLabels{
\includegraphics[height=0.32\textwidth]{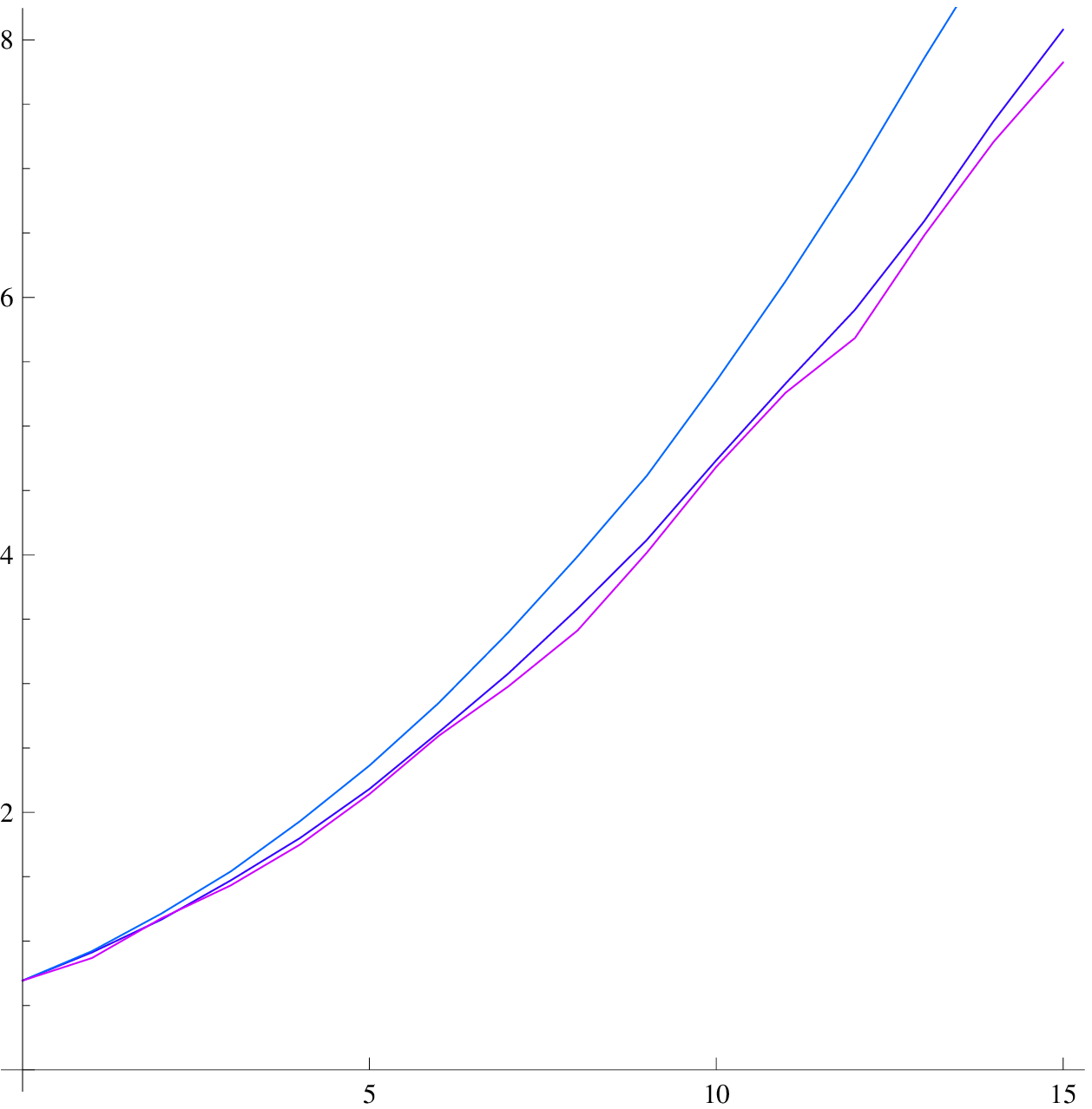}
\hskip 1 cm
\includegraphics[height=0.32\textwidth]{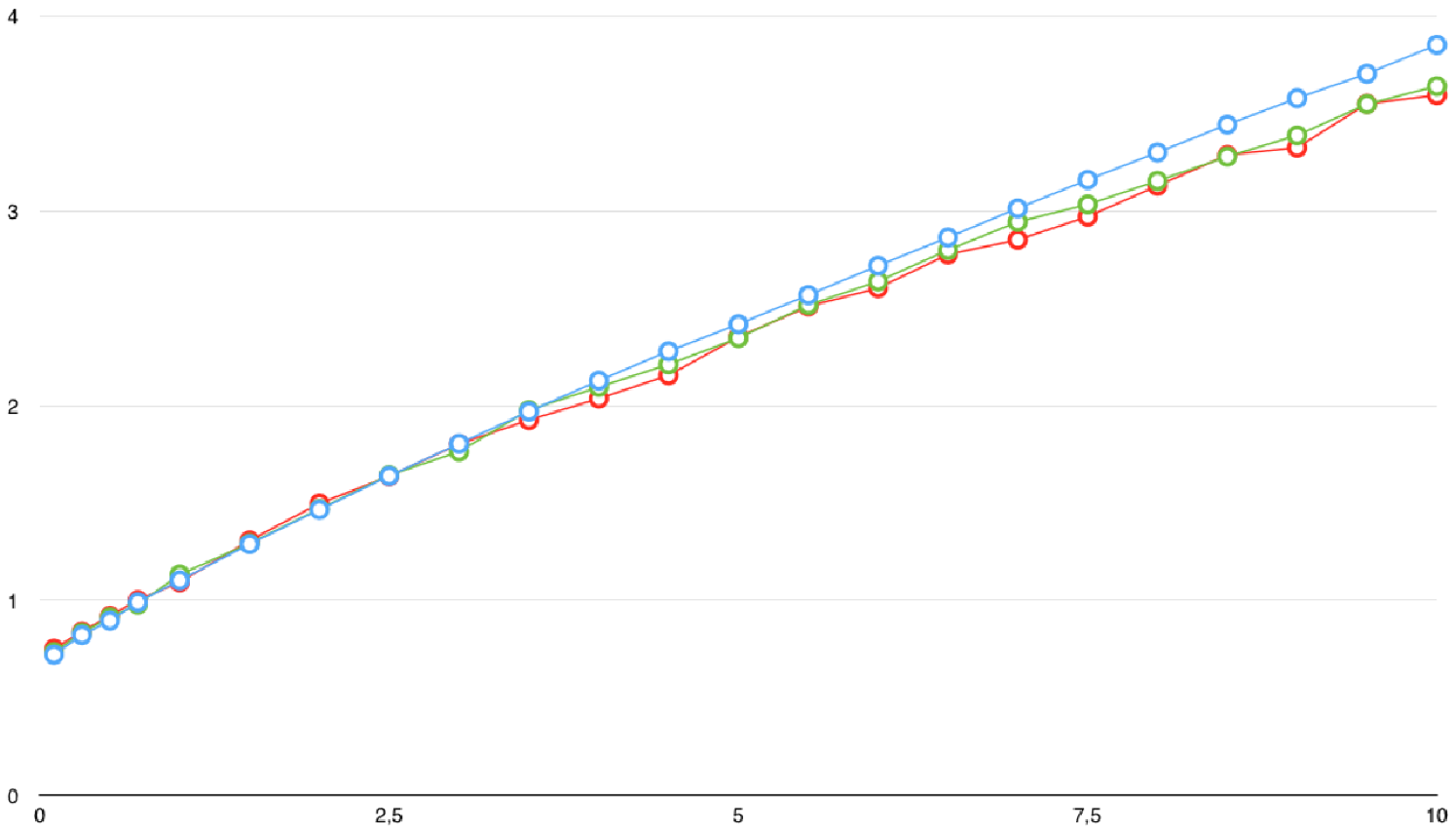}
}}
\caption{On the left, simulation results are shown for $-\log f(\lambda, [0,1]^2)$, with the near-critical percolation parameter varying from $\lambda=0$ to $1.5$, board sizes $n=10,100,500$. On the right, simulation results are shown for $-\log g(t, [0,1]^2)$ in dynamical percolation for scaled time going from $t=0$ to $10$, on board sizes $n=10,100,200$. In both cases, the values are lower and have more fluctuations as $n$ increases, since fewer simulation runs were feasible. The superexponential decay for $f(\lambda, [0,1]^2)$ is apparent, the subexponential decay for $g(t, [0,1]^2)$ is less so.}
\label{f.simu}
\end{figure}
%%%%%%%%%%%%%%%%%%%%%%%%%%%%%%%%

The proof of Theorem~\ref{t.tail} is very simple, given the results of \cite{garpetsch13,garpetsch} cited above, and standard percolation techniques proving exponential decay for certain connection probabilities, as can be found in \cite[Section 7]{nolin08}. This is somewhat similar to \cite{dumcop13}, where Duminil-Copin showed, again building on \cite{garpetsch13,garpetsch}, that the super-critical percolation Wulff shape is asymptotically circular, as the density $p$ approaches $p_c$.

\begin{proof}[Proof of Theorem~\ref{t.tail}]
By the scaling covariance~(\ref{e.scalecov}), we need to show that 
\begin{equation}\label{e.lam43}
f\big(-1,[0,\lambda^{4/3}]^2\big) = \exp\left(-\Theta\big(\lambda^{4/3}\big)\right)\,,
\end{equation}
as $\lambda\to\infty$. 

By \cite[Lemma 39 and (7.28)]{nolin08}, there exist $0<C_1\leq C_2<\infty$ such that for all $N$ and $p < p_c$,
\begin{equation}\label{e.long}
C_1 \exp \big(-C_2N/L(p)\big) \leq \Pbo{p}{\LR_{[0,N]^2} } \leq C_2 \exp \big(-C_1N/L(p)\big),
\end{equation}
where $L(p)$ is a correlation length; say,
\begin{equation}\label{e.corrlen}
L(p):=\inf\left\{n: \Pbo{p}{\LR_{[0,n]^2}} < 1/100 \right\}.
\end{equation}
Let us remark that~(\ref{e.long}) is nothing mysterious: the main reason for it is that $1/100$ is small enough so that a variation of the Peierls contour argument works on a renormalized lattice with mesh $L(p)$.

On the other hand, it is well-known (see, e.g., Proposition 34 in  \cite{nolin08}) that we have $L(p_c - r(n)) = \Theta(n)$, with the constant factors in $\Theta$ depending, of course, on $1/100$ in definition~(\ref{e.corrlen}). Now take $p=p_c-r(n)$ and $N =\lambda^{4/3}n$ in~(\ref{e.long}), then send $n\to\infty$ to see that~(\ref{e.lam43}) holds.
\end{proof}

\subsection{What about high dimensions?}

It is proved in Theorem~\ref{thm:newLIMITS} of the main text that for any probability distribution function $F(\cdot)$ there exists a sequence of monotone Boolean functions $\{f_n\}$ and some parameters $p_n,b_n$ such that $\Pbo{p_n+\lambda b_n}{f_n=1}\to F(\lambda)$ as $n\to\infty$, for all $\lambda\in\R$. However, all ``natural'' examples of limit distributions $F(\cdot)$ found so far have exponential or superexponential tails. The exponent $4/3$ in the superexponential decay $\exp(-|\lambda|^{4/3})$ of the previous section appears to be a rather direct consequence of the planar correlation length exponent being $\nu=4/3$ (i.e., the correlation length (\ref{e.corrlen}) satisfies $L(p)=|p-p_c|^{-4/3+o(1)}$ as $p\nearrow p_c$) and of the critical window being $r(n)=n^{-3/4+o(1)}$. Therefore, one might hope (as the present author did in the first version of this appendix) that crossing events in near-critical percolation on $\Z^d$, where $d$ is high enough so that already mean-field behaviour takes place, with $\nu=1/2 < 1$, could have subexponential tails. However, as we will briefly explain, this turns out to be very naive, and we expect now that the lower tail in high dimension is exponential in $|\lambda|$, while the upper tail is doubly exponential, similarly to some of the examples in Section~\ref{s.majority} of the main text.% (or Gaussian, depending on the exact question).

By high-dimensional percolation we will mean $d$ high enough so that the critical two-point connectivity function already scales like Green's function for simple random walk:
\begin{equation}\label{e.Green}
\Pbo{p_c}{\underline x\llra \underline y}=\Theta\big(\|\underline x-\underline y\|^{2-d} \big).
\end{equation}
This is conjectured to be the case for $d>6$, proved for $d\geq 19$ in \cite{HaraSlade90}, and recently for $d\geq 11$ in \cite{FivdH}. From now on, we will always assume that this mean-field behavior holds. Two-point connectivity can also be used to define a near-critical correlation length  $\xi(p)$:
$$
\Pbo{p}{\underline x\llra \underline y\hskip 0.2 cm \not\hskip -0.2 cm \llra\infty}=\exp\Big(-\Theta\big(\|\underline x-\underline y\| / \xi(p)\big)\Big)\,,
$$
which makes sense both for $p<p_c$ and $p>p_c$. In the planar case, it is known that $\xi(p)\asymp L(p)$, with $L(p)$ defined in  (\ref{e.corrlen}); see \cite[Theorem 33]{nolin08}. The mean-field value of the exponent in $\xi(p)=|p-p_c|^{-\nu+o(1)}$ is $\nu=1/2$, proved for $p<p_c$ in \cite{hara90}. In analogy with the 2-dimensional case, this may suggest that the critical window for left-to-right crossing in a box of side-length $n$ is of size $n^{-2+o(1)}$. However, there are some huge differences between planar and high-dimensional percolation, hence we have to examine this more closely. 

The heuristic picture of high-dimensional percolation, closely related to~(\ref{e.Green}), is that the critical cluster is like a critical Galton-Watson tree, embedded into $\Z^d$ as a branching random walk. A precise formulation of this is that the scaling limit of the Incipient Infinite Cluster is Integrated Infinite Canonical Super-Brownian Motion. This has been proved only very partially \cite{HaraSladeISBE,vdhofjar04,HHHM}; see \cite{SladeAMS} for a nice introduction into the limit object, and \cite{heyvdhof} for a survey. Therefore, in order to understand near-critical high-dimensional percolation, one could first look at near-critical Galton-Watson trees. One can show that, for any offspring distribution with mean $1-\lambda\eps$ and finite variance,
\begin{equation}\label{e.GWtree}
\Pbo{\mathsf{GW}(1-\lambda\eps)}{o \llra 1/\eps} =  \Theta(\eps) \,  \exp\big(-\Theta(\lambda)\big),\qquad 0 \leq \lambda < 1/\eps\,,
\end{equation}
where $\{o \llra 1/\eps\}$ denotes the event that the tree reaches generation $1/\eps$. 
%Similarly,
%\begin{equation}\label{e.GWtree.super}
%\Pbo{\mathsf{GW}(1+\lambda\eps)}{o \llra 1/\eps} =  \Theta(\lambda\eps),\qquad 0 \leq \lambda < 1/\eps\,.
%\end{equation}
There are several ways to prove this estimate; one is to consider a depth-first type exploration process of the tree, and use martingale calculations for the resulting integer-valued random walk, as in \cite{NachPer}. Now, accepting the heuristic picture above, and using that a random walk path in $\Z^d$ with $1/\eps$ steps goes to a Euclidean distance about $\sqrt{1/\eps}$ with large probability, the estimate~(\ref{e.GWtree}) suggests that, for near-critical percolation on $\Z^d$,
\begin{equation}\label{e.GWperc}
\PBo{p_c-\lambda\eps}{o \llra \partial B_{\sqrt{1/\eps}}(o)} =  \Theta(\eps) \,  \exp\big(-\Theta(\lambda)\big),\qquad 0 \leq \lambda < p_c/\eps\,,
\end{equation}
where $B_r(o)$ denotes the ball of Euclidean radius $r$. Note here that, in principle, the cluster could reach to a large distance not only by the embedded critical tree having a large radius, but also by the embedding reaching unusually far. However, changing the speed exponent $1/2$ has a probability cost that is exponential in the number of steps, $1/\eps$, hence this strategy cannot improve on~(\ref{e.GWperc}).

This was just heuristics, and in fact, only the critical case $\lambda=0$ of~(\ref{e.GWperc}) has been proved \cite{KozNachArms}: that is, the 1-arm exponent is 2. One approach to build an actual proof could be that the backbone of high-dimensional IIC is known to have a linear ``chain of sausages'' structure, with a linear number of pivotal edges, without too long sausages between them \cite{vdhofjar04,KozNachAO,HHHM}, hence the IIC must have a tree-like structure; also, its embedding into $\Z^d$ is known to have 4-dimensional features \cite{vdhofjar04,hhhRW}; these support the view that it is a large critical tree embedded as a branching random walk. Then, lowering the percolation density by $\lambda\eps$, the probability of keeping the connection, which exists at criticality with probability $\Theta(\eps)$, is at most the probability of not closing any of the pivotals, and should in fact be comparable to that: about $(1-\lambda\eps)^{\Theta(1/\eps)} \asymp \exp\big(-\Theta(\lambda)\big)$. This yields~(\ref{e.GWperc}).

We would like to use~(\ref{e.GWperc}) to understand the critical window of left-right crossing in a large box~$[n]^d$. One key difference from the planar case, observed in \cite{aizenman97}, is that at the critical density $p_c(\Z^d)$, with large probability there is already a large number of disjoint crossings. More precisely, using the critical two-point function~(\ref{e.Green}), we get that the expected number of vertices $y$ on the right side of $[n]^d$ connected to a given vertex $x$ on the left is on the order of $n^{2-d} n^{d-1}=n$. Alternatively, by the $\lambda=0$ case of~(\ref{e.GWperc}), the probability of $x$ being connected to the right side is about $n^{-2}$, and conditioned on this event, the cluster should look like a conditioned critical branching random walk, having about $n^4$ vertices within Euclidean distance $n$, hence about $n^3$ vertices on the right side, altogether giving an expectation about $n^{-2} n^3 = n$. Furthermore, the expected number of disjoint clusters connecting the left and right sides is about $n^{d-1} n / n^{2\cdot 3} = n^{d-6}$, since there are $n^{d-1} n$ pairs of vertices $x$ and $y$ that are connected to each other, with each cluster having $n^3$ possible $x$'s and $n^3$ possible $y$'s. Indeed, it was proved in \cite{aizenman97} that with probability tending to 1, there are at least order $n^{d-6}$ disjoint connections, with the possibility of having some more spanning clusters that are thinner than $4$-dimensional.

In particular, $p_c(\Z^d)$ is already outside of the critical window for left-to-right crossing. When, in order to find the critical window,  we start decreasing $p$ from $p_c(\Z^d)$, by~(\ref{e.GWperc}), the expected number of disjoint crossings should be about
\begin{equation}\label{e.GWpercExpect}
\EBo{p_c-\frac{\lambda}{n^2}}{\#\big\{\textrm{disjoint crossings from left to right in }[n]^d\big\}} =  \Theta(n^{d-6}) \,  \exp\big(-\Theta(\lambda)\big)\,;
\end{equation}
the effect of the slight subcriticality on the size of the conditioned tree, and hence on the factor $n^{d-6}$ should be only polynomial in $\lambda$, which is negligible compared to $\exp(-\lambda)$. To actually prove~(\ref{e.GWpercExpect}) or something a little weaker, instead of using~(\ref{e.GWperc}), it might be easier to start directly from Aizenman's proof \cite{aizenman97}.

Now, the finite size critical density $p_c(n)$, where the probability of having a left-to-right crossing is exactly $1/2$, should be around a value $p_c - \lambda / n^2$ where the above expectation~(\ref{e.GWpercExpect}) is about 1, and the critical window should be where the expectation is independent of $n$. This gives the following:
\begin{conj}
For left-to-right crossing in percolation on $[n]^d$, $d>6$, the critical density is
$$
p_c(n)= p_c(\Z^d) - (c_d + o(1)) \frac{\log n}{n^2} \,,
$$
with $c_d>0$. The critical window is
$$
p = p_c(n) + \frac{\lambda}{n^2},\qquad \lambda\in(-\infty,\infty)\,.
$$ 
In this window, the number of disjoint crossings should be approximately Poisson, with mean $\exp(\Theta(\lambda))$. Hence, for a very negative $\lambda$, the probability of a left-to-right crossing should be $\exp(-\Theta(|\lambda|))$, while, for a large positive $\lambda$, the probability of having no left-to-right crossing should be $\exp\big(-\exp(\Theta(\lambda))\big)$.
\end{conj}

This location of the window is also conjectured, independently, by Gady Kozma \cite{Gady}, while a Gumbel limit distribution is confirmed by computer simulations of Eren Metin El\c{c}i \cite{Eren}.

We also note that an upper bound $O(n^{-2/3})$ on the width of the critical window can in fact be proved via the following argument using randomized algorithms, partly due to Gady Kozma. 

Let $f_n: \{-1,1\}^{[n]^d} \lora \{-1,1\}$ be the $\pm1$-valued indicator function of the left-to-right crossing event in $[n]^d$, at a density $p$ in the near-critical $\eps$-window: $\Pso{p}{f_n=1}\in (\eps,1-\eps)$, for some fixed $\eps>0$. Note that, for $n$ large enough, we have $1/(2d) < p < p_c(\Z^d)$. 

Consider now the following randomized version of the algorithm used in \cite[Section 4]{BKSch} to determine the value of $f_n$. Choose uniformly at random a coordinate hyperplane separating the left and right faces. Now explore the clusters of all the sites in this hyperplane to determine whether there is a crossing. This way, each bit (site or edge) of the box is explored with probability at most $O(n^{-2/3})$: either the bit is within distance $n^{1/3}$ of the chosen hyperplane, which has probability $n^{-2/3}$, or it will be queried only if in a cluster of radius at least $n^{1/3}$, which has probability $O(n^{-2/3})$, since the critical one-arm exponent is 2, and we are at density $p < p_c(\Z^d)$.

In the wording of \cite{SchSt}, the above algorithm has revealment $O(n^{-2/3})$: the probability $\delta_i$ of revealing any bit $i$ is $O(n^{-2/3})$. By the main inequality of \cite{OSSS}, see also \cite[Theorem XII.36]{garste}, for the total influence we get
\begin{equation}\label{e.OSSS}
\Var_p(f_n) \leq 4p(1-p) \sum_{i\in [n]^d}  \delta_i \, I^p_i (f_n) \leq O(n^{-2/3})\, I^p(f_n)\,.
\end{equation}
Being in the $\eps$-window, we have $4\eps(1-\eps) \leq \Var_p(f_n)$, hence~(\ref{e.OSSS}) gives $c(\eps) \, n^{2/3} \leq I^p(f_n)$, for some $c(\eps) > 0$. By Russo's formula, this implies that the near-critical $\eps$-window for $f_n$ has width at most $C(\eps) \, n^{-2/3}$, as we claimed.

Let us mention that another way to do this argument, yielding a worse exponent, would have been via noise sensitivity and Fourier analysis; for background, see \cite{garste}. By \cite[Theorem 1.8]{SchSt} (or more precisely, by its straightforward extension from density $p=1/2$ to general $p$ values bounded away from 0 and 1), the revealment $O(n^{-2/3})$ implies that $f_n$ is sensitive to any noise $\gg n^{-1/3}$. In terms of the Fourier spectral sample $\Spec^p(f_n)$, defined by
$$
\Pb{\Spec^p(f_n)=S} := \frac{\widehat{f_n}(S)^2}{\|f_n\|_2^2}=\widehat{f_n}(S)^2\,,\qquad S\subset [n]^d\,,
$$
where $\widehat{f}(S):=\Ebo{p}{f\, \chi^p_{S}}$ is the Fourier coefficient corresponding to the orthonormal basis
$$
\chi^p_S(\omega) := \prod_{i\in S} \omega_i \left(\frac{1-p}{p}\right)^{\omega_i/2} ,\qquad \omega\in \{-1,1\}^{[n]^d}
$$
for $L^2\big(\{-1,1\}^{[n]^d},\E_p\big)$, this noise sensitivity can be written as follows: for any $\eps$-window and any $\delta>0$, if $\kappa>0$ is small enough, then
\begin{equation}\label{e.Spec}
\Pb{0 < |\Spec^p(f_n)| < \kappa \, n^{1/3}} < \delta\,.
\end{equation}
On the other hand, $\Pb{\Spec^p(f_n)=\emptyset}=\Ebo{p}{f_n}^2 \leq (1-2\eps)^2$, since we are in the $\eps$-window.
Combined with~(\ref{e.Spec}), for $\delta$ sufficiently small, we get that the total influence, for $p$ in the $\eps$-window, satisfies
$$
I^p(f_n) = \Eb{|\Spec^p(f_n)|} \geq  c(\eps) \, n^{1/3},\qquad  c(\eps) > 0\,,
$$
giving that the width of the $\eps$-window is at most $C(\eps)\, n^{-1/3}$.
\bigskip

Finally, one might prefer to deal with transitive Boolean functions only, hence would want to consider high-dimensional percolation on tori.
 Another observation of \cite{aizenman97} was that finite size boundary effects become important here, and the cluster structure in a torus is different from the cluster structure in a box. In particular, for percolation on the torus, already Erd\H{o}s-R\'enyi random graph asymptotics take place: the largest critical cluster has size of order $n^{2d/3}$, and the critical window should be $n^{-d/3}$. (Nevertheless, large clusters are still ``four-dimensional'', similarly to the box case; in particular, their Euclidean diameter, when pulled back to the universal cover of the torus, is $n^{d/6}$.) A large part of this conjectured basic near-critical picture has already been proved; see \cite{borchavdhofslaspe05a,borchavdhofslaspe05b,heyvdhof07,heyvdhof11,vdhofsap14}.
 
The above-mentioned critical window $n^{-d/3}$ should hold for most natural monotone events on the torus: e.g., for the existence of a non-contractible cluster, or a cluster of pulled-back Euclidean diameter $n^{d/6}$, or a cluster of size $n^{2d/3}$. However, the near-critical tails should be different. At the lower end of the window, with very negative $\lambda$, to get a cluster with large Euclidean diameter $n^{d/6}$ (independent of $\lambda$), the near-critical GW tree diameter asymptotics~(\ref{e.GWtree}) should be relevant, yielding a tail $\exp(-\Theta(|\lambda|))$. To get a cluster with volume $n^{2d/3}$, the near-critical GW tree volume asymptotics should be relevant, which can be shown to be $\exp(-\Theta(|\lambda|^2))$. A similar difference for the Erd\H{o}s-R\'enyi random graphs was pointed out in \cite{Luczak}. Regarding the tail for the existence of a non-contractible cluster, we will not make a guess. But in any case, it seems unlikely that natural monotone events will have subexponential tails.

For a recent survey of high-dimensional percolation and random graphs, see \cite{heyvdhof}.

\section*{Acknowledgements}
This work was begun while DA was visiting the Department for Mathematical Sciences, Gothenburg, 
to which he is grateful for their hospitality.
The main authors thank Anders Martinsson for jointly proving part~\emph{b)} of Theorem~\ref{thm:newLIMITS} and for other useful comments. They further thank Bo Berndtsson, Christophe Garban and Benjy Weiss
for some discussions, and an anonymous referee for a number of comments on the manuscript.

The appendix author is grateful to Bal\'azs R\'ath for an eye-opening discussion and comments on the manuscript, and to the referee for pointing out that the proof of~(\ref{e.lam43}) can basically be found already in \cite{nolin08}. GP is also grateful to Christophe Garban and Alan Hammond for many discussions on near-critical and dynamical percolation, and to Remco van der Hofstad, Tim Hulshof, and Gady Kozma for conversations about the high-dimensional case.

DA was during this work supported by postdoctoral grants
150804/2012-1 from the Brazilian CNPq and 637-2013-7302 from the Swedish Research Council.
JES also acknowledges the support of the Swedish Research Council and the Knut and Alice Wallenberg Foundation.
GP was partially supported by the Hungarian National Science Fund OTKA grant K109684 and by the MTA R\'enyi Institute ``Lend\"ulet'' Limits of Structures Research Group.

%\bibliographystyle{abbrv}
%\bibliography{bib-bf}

\begin{thebibliography}{10}

\bibitem{aizenman97}
M.~Aizenman.
\newblock On the number of incipient spanning clusters.
\newblock {\em Nuclear Phys. B}, 485(3):551--582, 1997.

\bibitem{BKSch}
I.~Benjamini, G.~Kalai, and O.~Schramm. 
\newblock Noise sensitivity of Boolean functions and applications to percolation. 
\newblock {\em Inst. Hautes \'Etudes Sci. Publ. Math.}, 90:5--43, 1999.

\bibitem{bollobas01}
B.~Bollob{\'a}s.
\newblock {\em Random graphs}, volume~73 of {\em Cambridge Studies in Advanced
  Mathematics}.
\newblock Cambridge University Press, Cambridge, second edition, 2001.

\bibitem{boltho87}
B.~Bollob{\'a}s and A.~Thomason.
\newblock Threshold functions.
\newblock {\em Combinatorica}, 7(1):35--38, 1987.

\bibitem{borchavdhofslaspe05a}
C.~Borgs, J.~T. Chayes, R.~van~der Hofstad, G.~Slade, and J.~Spencer.
\newblock Random subgraphs of finite graphs. {I}. {T}he scaling window under
  the triangle condition.
\newblock {\em Random Structures Algorithms}, 27(2):137--184, 2005.

\bibitem{borchavdhofslaspe05b}
C.~Borgs, J.~T. Chayes, R.~van~der Hofstad, G.~Slade, and J.~Spencer.
\newblock Random subgraphs of finite graphs. {II}. {T}he lace expansion and the
  triangle condition.
\newblock {\em Ann. Probab.}, 33(5):1886--1944, 2005.

\bibitem{boukal97}
J.~Bourgain and G.~Kalai.
\newblock Influences of variables and threshold intervals under group
  symmetries.
\newblock {\em Geom. Funct. Anal.}, 7(3):438--461, 1997.

\bibitem{dumcop13}
H.~Duminil-Copin.
\newblock Limit of the {W}ulff crystal when approaching critically for site
  percolation on the triangular lattice.
\newblock {\em Electron. Commun. Probab.}, 18:no. 93, 9, 2013.

\bibitem{durrett10}
R.~Durrett.
\newblock {\em Probability: theory and examples}.
\newblock Cambridge Series in Statistical and Probabilistic Mathematics.
  Cambridge University Press, Cambridge, fourth edition, 2010.

\bibitem{Eren}
E.~M.~El\c{c}i.
\newblock Personal communication, 2016.

\bibitem{erdren60}
P.~Erd{\H{o}}s and A.~R{\'e}nyi.
\newblock On the evolution of random graphs.
\newblock {\em Magyar Tud. Akad. Mat. Kutat\'o Int. K\"ozl.}, 5:17--61, 1960.

\bibitem{FivdH}
R.~Fitzner and R.~van~der~Hofstad.
Nearest-neighbor percolation function is continuous for $d>10$.
\newblock Preprint, see \emph{arXiv:\allowbreak 1506.07977}

\bibitem{friedgut99}
E.~Friedgut.
\newblock Sharp thresholds of graph properties, and the {$k$}-sat problem.
\newblock {\em J. Amer. Math. Soc.}, 12(4):1017--1054, 1999.
\newblock With an appendix by Jean Bourgain.

\bibitem{frikal96}
E.~Friedgut and G.~Kalai.
\newblock Every monotone graph property has a sharp threshold.
\newblock {\em Proc. Amer. Math. Soc.}, 124(10):2993--3002, 1996.

\bibitem{garpetsch13}
C.~Garban, G.~Pete, and O.~Schramm.
\newblock Pivotal, cluster, and interface measures for critical planar
  percolation.
\newblock {\em J. Amer. Math. Soc.}, 26(4):939--1024, 2013.

\bibitem{garpetsch}
C.~Garban, G.~Pete, and O.~Schramm.
\newblock The scaling limits of near-critical and dynamical percolation.
\newblock Preprint, see \emph{arXiv:\allowbreak 1305.5526}.

\bibitem{garste12}
C.~Garban and J.~E. Steif.
\newblock Noise sensitivity and percolation.
\newblock In {\em Probability and statistical physics in two and more
  dimensions}, volume~15 of {\em Clay Math. Proc.}, pages 49--154. Amer. Math.
  Soc., Providence, RI, 2012.

\bibitem{garste}
C.~Garban and J.~E. Steif.
\newblock \emph{Noise sensitivity of {B}oolean functions and percolation}.
\newblock Cambridge University Press, 2014.

\bibitem{grimmett99}
G.~Grimmett.
\newblock {\em Percolation}, volume 321 of {\em Grundlehren der Mathematischen
  Wissenschaften [Fundamental Principles of Mathematical Sciences]}.
\newblock Springer-Verlag, Berlin, second edition, 1999.

\bibitem{hammospet12}
A.~Hammond, E.~Mossel, and G.~Pete.
\newblock Exit time tails from pairwise decorrelation in hidden {M}arkov
  chains, with applications to dynamical percolation.
\newblock {\em Electron. J. Probab.}, 17:no. 68, 16, 2012.

\bibitem{hara90}
T.~Hara.
\newblock Mean-field critical behaviour for correlation length for percolation in high dimensions.
\newblock {\em Probab. Theory Related Fields}, 86(3):337--385, 1990.

\bibitem{HaraSlade90}
T.~Hara and G.~Slade. 
\newblock Mean-field critical behaviour for percolation in high dimensions.
\newblock {\em Commun. Math. Phys.}, 128:333--391, 1990.

\bibitem{HaraSladeISBE}
T.~Hara and G.~Slade. 
\newblock The scaling limit of the incipient infinite cluster in high-dimensional percolation, II. Integrated super-Brownian excursion.\newblock {\em J. Math. Phys.}, 41(3):1244--1293, 2000.

\bibitem{HHHM}
M.~Heydenreich, R.~van~der~Hofstad, T.~Hulshof, and G.~Miermont.
Backbone scaling limit of the high-dimensional IIC. {\it In preparation}.

\bibitem{heyvdhof}
M.~Heydenreich and R.~van~der Hofstad.
\newblock Progress in high-dimensional percolation and random graphs.
\newblock Lecture notes for the CRM-PIMS Summer School in Probability 2015, see
  \url{http://www.win.tue.nl/~rhofstad/survey-high-d-percolation.pdf}.

\bibitem{heyvdhof07}
M.~Heydenreich and R.~van~der Hofstad.
\newblock Random graph asymptotics on high-dimensional tori.
\newblock {\em Comm. Math. Phys.}, 270(2):335--358, 2007.

\bibitem{heyvdhof11}
M.~Heydenreich and R.~van~der Hofstad.
\newblock Random graph asymptotics on high-dimensional tori {II}: volume,
  diameter and mixing time.
\newblock {\em Probab. Theory Related Fields}, 149(3-4):397--415, 2011.

\bibitem{hhhRW}
M.~Heydenreich, R.~van~der Hofstad, and T.~Hulshof.
\newblock Random walk on the high-dimensional IIC.
\newblock {\em Comm. Math. Phys.}, 329(1):57--115, 2014.

%\bibitem{heyvdhofhul14}
%M.~Heydenreich, R.~van~der Hofstad, and T.~Hulshof.
%\newblock High-dimensional incipient infinite clusters revisited.
%\newblock {\em J. Stat. Phys.}, 155(5):966--1025, 2014.

\bibitem{vdhofjar04}
R.~{\noopsort{Hofstad}}van~der Hofstad and A.~A. J{\'a}rai.
\newblock The incipient infinite cluster for high-dimensional unoriented
  percolation.
\newblock {\em J. Statist. Phys.}, 114(3-4):625--663, 2004.

\bibitem{vdhofsap14}
R.~{\noopsort{Hofstad}}van~der Hofstad and A.~Sapozhnikov.
\newblock Cycle structure of percolation on high-dimensional tori.
\newblock {\em Ann. Inst. Henri Poincar\'e Probab. Stat.}, 50(3):999--1027,
  2014.

\bibitem{kahkallin88}
J.~Kahn, G.~Kalai, and N.~Linial.
\newblock The influence of variables on {B}oolean functions.
\newblock In {\em 29th Annual Symposium on Foundations of Computer Science},
  pages 68--80, 1988.

\bibitem{kalsaf06}
G.~Kalai and S.~Safra.
\newblock Threshold phenomena and influence: perspectives from mathematics,
  computer science, and economics.
\newblock In {\em Computational complexity and statistical physics}, St. Fe
  Inst. Stud. Sci. Complex., pages 25--60. Oxford Univ. Press, New York, 2006.

\bibitem{kesten87}
H.~Kesten.
\newblock Scaling relations for {$2$}{D}-percolation.
\newblock {\em Comm. Math. Phys.}, 109(1):109--156, 1987.

\bibitem{Gady}
G.~Kozma. 
\newblock Personal communication, 2016.

\bibitem{KozNachAO}
G.~Kozma and A.~Nachmias. 
\newblock The Alexander-Orbach conjecture holds in high dimensions. 
\newblock {\em Invent. Math.}, 178:635--654, 2009.

\bibitem{KozNachArms}
G.~Kozma and A.~Nachmias. 
\newblock Arm exponents in high dimensional percolation. 
\newblock {\em J. Amer. Math. Soc.}, 24(2):375--409, 2011.

\bibitem{langlands05}
R.~P. Langlands.
\newblock The renormalization fixed point as a mathematical object.
\newblock In {\em Twenty years of {B}ialowieza: a mathematical anthology},
  volume~8 of {\em World Sci. Monogr. Ser. Math.}, pages 185--216. World Sci.
  Publ., Hackensack, NJ, 2005.

\bibitem{lanlaf94}
R.~P. Langlands and M.-A. Lafortune.
\newblock Finite models for percolation.
\newblock In {\em Representation theory and analysis on homogeneous spaces
  ({N}ew {B}runswick, {NJ}, 1993)}, volume 177 of {\em Contemp. Math.}, pages
  227--246. Amer. Math. Soc., Providence, RI, 1994.

\bibitem{Luczak}
T.~{\L}uczak.
\newblock Random trees and random graphs. 
\newblock {\em Rand. Struc. Alg.}, 13: 485--500,  1998.

\bibitem{mosodo03}
E.~Mossel and R.~O'Donnell.
\newblock On the noise sensitivity of monotone functions.
\newblock {\em Random Structures Algorithms}, 23(3):333--350, 2003.

\bibitem{NachPer}
A.~Nachmias and Y.~Peres.
\newblock The critical random graph, with martingales.
\newblock {\em Israel J. Math.}, 176(1): 29--41, 2010.

\bibitem{nolin08}
P.~Nolin.
\newblock Near-critical percolation in two dimensions.
\newblock {\em Electron. J. Probab.}, 13:no. 55, 1562--1623, 2008.

\bibitem{OSSS}
R.~O'Donnell, M.~Saks, O.~Schramm, and R.~Servedio.
\newblock Every decision tree has an influential variable. 
\newblock {\em 46th Annual IEEE Symposium on Foundations of Computer Science (FOCS'05)}, pp.~31--39, 2005.

\bibitem{pete-talk}
G.~Pete.
\newblock How long could it take to establish a crossing in dynamical
  percolation?
\newblock Talk at MFO, Oberwolfach, September 2012, see
  \url{http://www.math.bme.hu/~gabor/DynExitOwolfach.pdf}.
  
\bibitem{rossignol07}
R.~Rossignol.
\newblock Arbitrary threshold widths for monotone, symmetric properties.
\newblock {\em J. Appl. Probab.}, 44(1): 164--180, 2007.

\bibitem{russo82}
L.~Russo.
\newblock An approximate zero-one law.
\newblock {\em Z. Wahrsch. Verw. Gebiete}, 61(1):129--139, 1982.

\bibitem{schsmi11}
O.~Schramm and S.~Smirnov.
\newblock On the scaling limits of planar percolation.
\newblock {\em Ann. Probab.}, 39(5):1768--1814, 2011.
\newblock With an appendix by Christophe Garban.

\bibitem{SchSt}
O.~Schramm and J.~Steif.
\newblock Quantitative noise sensitivity and exceptional times for percolation. 
\newblock {\em Ann. Math.}, 171(2):619--672, 2010.

\bibitem{SladeAMS}
G.~Slade.
\newblock Scaling limits and super-Brownian motion.
 \newblock {\em Notices AMS}, 49(9): 1056--1067, 2002.

\bibitem{smiwer01}
S.~Smirnov and W.~Werner.
\newblock Critical exponents for two-dimensional percolation.
\newblock {\em Math. Res. Lett.}, 8(5-6):729--744, 2001.

\bibitem{talagrand94}
M.~Talagrand.
\newblock On {R}usso's approximate zero-one law.
\newblock {\em Ann. Probab.}, 22(3):1576--1587, 1994.

\bibitem{werner09}
W.~Werner.
\newblock Lectures on two-dimensional critical percolation.
\newblock In {\em Statistical mechanics}, volume~16 of {\em IAS/Park City Math.
  Ser.}, pages 297--360. Amer. Math. Soc., Providence, RI, 2009.

\end{thebibliography}
\newcommand{\noopsort}[1]{}\def\cprime{$'$}

\medskip

{\small
\noindent
{\sc Daniel Ahlberg\\
Instituto Nacional de Matem\'atica Pura e Aplicada\\
Estrada Dona Castorina 110, 22460-320 Rio de Janeiro, Brasil\\
Department of Mathematics, Uppsala University\\
SE-75106 Uppsala, Sweden}\\
\url{www.impa.br/~ahlberg}\\
\texttt{ahlberg@impa.br}\\

\noindent
{\sc Jeffrey E.\ Steif\\
Mathematical Sciences, Chalmers University of Technology\\
Mathematical Sciences, University of Gothenburg\\
SE-41296 Gothenburg, Sweden}\\
\url{http://www.math.chalmers.se/~steif}\\
\texttt{steif@chalmers.se}\\

\noindent
{\sc G\'abor Pete\\
R\'enyi Institute, Hungarian Academy of Sciences, Budapest\\
13-15 Reáltanoda u., 1053 Budapest, Hungary\\
Institute of Mathematics, Budapest University of Technology and Economics\\
1 Egry J\'ozsef u., 1111 Budapest, Hungary}\\
\url{http://www.math.bme.hu/~gabor}\\
\texttt{gabor@math.bme.hu}
}

\end{document}